\documentclass{amsart} 
\usepackage{amsmath, amssymb, amsfonts}
\usepackage{hyperref, cite, url}
\usepackage{graphicx, xcolor}
\usepackage[all]{xy}
\theoremstyle{plain}

\newtheorem{thm}{Theorem}[section]
\newtheorem{proposition}[thm]{Proposition}
\newtheorem{lemma}[thm]{Lemma}
\newtheorem{corollary}[thm]{Corollary}
\theoremstyle{definition}
\newtheorem{definition}[thm]{Definition}

\newtheorem{example}[thm]{Example}
\newtheorem{remark}[thm]{Remark}
     \def\Ker{\operatorname{Ker}}
\def\Hom{\operatorname{Hom}}   \def\ri{\operatorname{ri}}
\def\Supp{\operatorname{Supp}} \def\Ann{\operatorname{Ann}}
\def\HF{\operatorname{HF}}     \def\HP{\operatorname{HP}}
\def\rank{\operatorname{rank}} 
\newcommand{\bbQ}{\ensuremath{\mathbb Q}}
\newcommand{\bbZ}{\ensuremath{\mathbb Z}}
\newcommand{\bbN}{\ensuremath{\mathbb N}}
\newcommand{\bbP}{\ensuremath{\mathbb P}}
\newcommand{\bbX}{\ensuremath{\mathbb X}}
\newcommand{\bbY}{\ensuremath{\mathbb Y}}
\newcommand{\bbW}{\ensuremath{\mathbb W}}

\newcommand{\fq}{\mathfrak{q}}
\newcommand{\fm}{\mathfrak{m}}
\newcommand{\fC}{\mathfrak{C}}

\newcommand{\fQ}{\mathfrak{Q}}
\newcommand{\fF}{\mathfrak{F}}
\newcommand{\calO}{{\mathcal{O}}}
\newcommand{\calZ}{{\mathcal{Z}}}
\newcommand{\calA}{{\mathcal{A}}}
\newcommand{\calE}{{\mathcal{E}}}
\newcommand{\calV}{{\mathcal{V}}}
\newcommand{\calW}{{\mathcal{W}}}

\newcommand{\calC}{{\mathcal{C}}}

\newcommand{\calY}{{\mathcal{Y}}}

\def\ideal#1{\langle\, {#1}\,\rangle}

\begin{document}

\title[The K\"ahler different of a 0-dimensional scheme]
{The K\"ahler different of a 0-dimensional scheme}

\author{Le~Ngoc~Long}
\address[Le~Ngoc~Long]{
Department of Mathematics,
University of Education - Hue University\\
34 Le Loi, Hue, Vietnam
\newline
\hspace*{.5cm} \textrm{and}
Fakult\"{a}t f\"{u}r Informatik und Mathematik \\
Universit\"{a}t Passau, D-94030 Passau, Germany }
\email{lelong@hueuni.edu.vn}



\date{\today}

\dedicatory{}

\commby{L.N. Long}

\begin{abstract}
Given a 0-dimensional scheme $\bbX$ in the projective $n$-space 
$\mathbb{P}^n_K$ over a field $K$, we are interested in studying 
the K\"ahler different of~$\bbX$ and its applications.
Using the K\"ahler different, we characterize the generic position and 
Cayley-Bacharach properties of $\bbX$ in several certain cases. 
When $\bbX$ is in generic position, we prove a generalized version of the
Ap\'{e}ry-Gorenstein-Samuel theorem about arithmetically Gorenstein schemes. 
We also characterize 0-dimensional complete intersections 
in terms of the K\"ahler different and the Cayley-Bacharach property. 
\end{abstract}
\maketitle

\section{Introduction}

Let $K$ be a field, and let $\bbX$ be a non-empty 0-dimensional subscheme 
$\bbX$ of the projective $n$-space $\bbP^n_K$ over $K$.
The homogeneous coordinate ring of $\bbP^n_K$ is the polynomial ring 
$P=K[X_0,...,X_n]$ equipped with the standard grading, and the homogeneous
vanishing ideal of $\bbX$ in~$P$ is denoted by~$I_\bbX$. Then
the homogeneous coordinate ring of $\bbX$ is given by~$R=P/I_\bbX$.
Once and for all, we assume that the characteristic of $K$ satisfies 
either ${\rm char}(K)=0$ or ${\rm char}(K)>\deg(\bbX)$. 
Consequently, the coordinates of $\bbP^n_K$ are chosen such
that no closed point of $\bbX$ lies on the hyperplane 
$\calZ(X_0)$ and every reduced point of $\bbX$ is also smooth.
In particular, the image $x_0$ of $X_0$ in $R$ is a non-zerodivisor,
$\deg(\bbX)=\dim_K(R/\ideal{x_0-1})$, and 
$R$ is a graded-free $K[x_0]$-algebra of rank $\deg(\bbX)$.
The purpose of this paper is to study the algebraic structure of 
the K\"ahler different $\vartheta_\bbX$ of $\bbX$ 
(or of the algebra $R/K[x_0]$), and then use it to look at geometrical
properties of $\bbX$ such as generic position, Cayley-Bacharach 
and complete intersection properties. 
The K\"ahler different $\vartheta_\bbX$ is known as the initial 
Fitting ideal of the module of K\"ahler differentials of $R/K[x_0]$. 
It is a homogeneous ideal of $R$ generated by all maximal minors of the
Jacobian matrix of a homogeneous system of generators of $I_\bbX$. 

The K\"ahler different is a classical algebraic invariant which has been
studied in different contexts. Many structural properties
of an algebra are encoded in this invariant (see for instance 
\cite{EU2019, DK1999, IT2021, Kun1986, KLL2015, KL2017, Noe1950, SS1975}).
In \cite{Kun1960}, Kunz showed that the K\"ahler different is a subideal
of the Noether different and these differents agree with 
the Dedekind different for reduced complete intersections.
These three differents have resurfaced in the work of 
Eisenbud and Ulrich \cite{EU2019} on residual intersections.
Recently, Iyengar and Takahashi \cite{IT2021} have used the 
K\"ahler different to study annihilators of cohomology. 

In this paper, we initiate our study by determining the Hilbert polynomial
of the K\"ahler different $\vartheta_\bbX$ and giving bounds 
for its regularity index.
It turns out that one can detect a 0-dimensional locally complete 
intersection by looking at the value of the Hilbert polynomial 
of~$\vartheta_\bbX$. 
We say that $\bbX$ is \textit{in generic position} if 
$\HF_{\bbX}(i) = \min\{\, \deg(\bbX), \binom{n+i}{n} \,\}$
for all $i \in \bbZ$. This notion was introduced by 
Geramita and Orecchia in \cite{GO1981} for finite sets of $K$-rational 
points in $\bbP^n_K$ and has been investigated in different contexts 
(see \cite{HT2004, DK1999, KR2005, KS2016}). By looking at the 
Hilbert function of the K\"ahler different, 
we provide characterizations of the generic position and 
Cayley-Bacharach properties of a reduced 0-dimensional scheme $\bbX$. 
In particular, we show that $\bbX$ is in generic position 
with $\deg(\bbX)=\binom{n+r_{\bbX}}{n}$ if and only if 
$\vartheta_\bbX = \oplus_{i\ge nr_\bbX}R_i$,
which in turn is equivalent to the condition that 
$\bbX$ is a CB-scheme and $\vartheta_\bbX = \fF^n_{\widetilde{R}/R}$,
where $r_\bbX$ is the regularity index of $\HF_\bbX$, where 
$\widetilde{R}$ is the integral closure of $R$ 
in its full quotient ring, and where $\fF_{\widetilde{R}/R}$ 
is the conductor of $R$ in $\widetilde{R}$.
The notion of a Cayley-Bacharach scheme (CB-scheme) 
has a long and rich history (see \cite{KLR2019}).
Classically, a finite set of $K$-rational points $\bbX\subseteq\bbP^n_K$ 
is called a \textit{CB-scheme} if every hypersurface of degree $r_\bbX-1$ 
which contains all points of~$\bbX$ but one automatically contains 
the last point. 
Later, this notion was generalized for 0-dimensional schemes
over an arbitrary field in \cite{KLL2019, KLR2019}. 
Examples of work on this property can be seen in  
\cite{DGO1985, GH1978, GKR1993, GMT2010, Kre1992, KLL2019, KLLT2020, KLR2019}. 
Furthermore, it is well-known that every arithmetically Gorenstein scheme
(i.e., $R$ is a Gorenstein ring) is also a CB-scheme. 
An interesting result on arithmetically Gorenstein 
schemes is the Ap\'{e}ry-Gorenstein-Samuel theorem (see \cite{Bas1963})
that states: \textit{A reduced 0-dimensional scheme $\bbX\subseteq\bbP^n_K$ 
is arithmetically Gorenstein if and only if 
$\ell(\widetilde{R}/R)=\ell(R/\mathfrak{F}_{\widetilde{R}/R})$}.
In \cite{GO1981}, Geramita and Orrechia showed that
this theorem plays an important role for describing 
the occurrence of Gorenstein singularities of finite generic sets 
of $K$-rational points in $\bbP^n_K$. One of our main results is the
following generalized version of the Ap\'{e}ry-Gorenstein-Samuel theorem. 
\begin{thm}[Theorem~\ref{thm:Apery-Gorenstein-Samual}]
	Let $\bbX\subseteq \bbP^n_K$ be a $0$-dimensional locally Gorenstein 
	scheme in generic position.
	Then $\bbX$ is arithmetically Gorenstein if and only if 
	$\ell(\widetilde{R}/R) = \ell(R/\mathfrak{F}_{\widetilde{R}/R})$.
\end{thm}

Furthermore, we look at the question: 
\textit{if $\bbX$ is a CB-scheme, when is it a complete intersection?}
posed by Griffiths and Harris \cite{GH1978}. 
In $\bbP^2_K$, David and Maroscia \cite{DM1984} gave a possible answer 
to the question that $\bbX$ is a complete intersection if and only if 
it is a CB-scheme with symmetric Hilbert function. 
However, in the general case $n\ge 2$, the condition that $\bbX$ 
is a CB-scheme with symmetric Hilbert function is equivalent to that
$\bbX$ is arithmetically Gorenstein, as was showed in \cite{Kre1992}
for the scheme $\bbX$ over an algebraically closed field 
and in \cite{KLR2019} for $\bbX$ over an arbitrary field.
By looking more closely at the relation between 
the K\"ahler, Noether and Dedekind differents, we provide 
a positive answer to the above question for arbitrary $n\ge 1$:
\begin{thm}[Theorem~\ref{thm:CharacterizationCI-KDiff}]
Let $\bbX$ be a 0-dimensional scheme in $\bbP^n_K$. Then
the following conditions are equivalent:
\begin{enumerate}
\item[(a)] $\bbX$ is a complete intersection.
\item[(b)] $\bbX$ is a locally Gorenstein CB-scheme and 
$\HF_{\vartheta_\bbX}(r_\bbX) \ne 0$.
\item[(c)] $\bbX$ is a CB-scheme and 
$\HP_{\vartheta_\bbX} = \sum_{j=1}^s\dim_K K(p_j)$ and
$\HF_{\vartheta_\bbX}(r_\bbX) \ne 0$.
\end{enumerate}
\end{thm}
\noindent 
This theorem is a generalization of one of main results of~\cite{KL2017}
for smooth 0-dimensional schemes in $\bbP^n_K$. 

This paper is outlined as follows. In Section~\ref{Sec2},  
we recall the needed facts about maximal $p_j$-subschemes, 
sets of separators, and the K\"ahler different.
Especially, we determine the Hilbert polynomial of the K\"ahler different 
and bound its regularity index (see Proposition~\ref{prop_HF-Kdiff}) 
and characterize locally complete intersections (see Corollary~\ref{cor:LocallyCI}). 
In Section~\ref{Sec3}, we use the K\"ahler different
$\vartheta_\bbX$ to look at the generic position property as well 
as the Cayley-Bacharach property of $\bbX$ (see Propositions~\ref{prop:GenPos-KDiff-CBP} and \ref{propSec4.10}). 
We prove Theorem~\ref{thm:Apery-Gorenstein-Samual}, which is
a generalized version of Ap\'{e}ry-Gorenstein-Samuel theorem and 
discuss its consequence. 
In the final section, we prove a characterization of 0-dimensional
arithmetically Gorenstein schemes in terms of their Noether different 
(see Proposition~\ref{CharGorNoetDiff}) and
then prove Theorem~\ref{thm:CharacterizationCI-KDiff}.

Unless explicitly stated otherwise, we adhere to the definitions
and notation introduced in \cite{KR2005} and \cite{Kun1986}.
All examples in this paper were calculated by using
the computer algebraic system ApCoCoA~\cite{ApC}.

\medskip\bigbreak
\section{K\"ahler Differents of 0-Dimensional Schemes}
\label{Sec2}
 
Let $\bbX$ be a 0-dimensional scheme in $\bbP^n_K$, 
and let $I_\bbX\subseteq P=K[X_0,...,X_n]$ be the homogeneous vanishing 
ideal of~$\bbX$. Its homogeneous coordinate ring 
$R=P/I_\bbX$ is a 1-dimensional Cohen-Macaulay ring. 
The set of closed points of~$\bbX$ is called the {\it support}
of~$\bbX$ and denoted by $\Supp(\bbX) = \{p_1,\dots,p_s\}$.
Under the assumption of the characteristic of $K$, we always choose
the coordinates $\{X_0,...,X_n\}$ of $\bbP^n_K$ such that
$\Supp(\bbX) \cap \calZ(X_0)=\emptyset$.
By $x_i$ we denote the image of $X_i$ in $R$ for $i=0,...,n$.
Then $x_0$ is a non-zerodivisor of $R$.
The coordinate ring $S$ of $\bbX$ in the affine space 
$\mathbb{A}^n\cong D_+(X_0)$ is $S= R/\ideal{x_0-1}$, and 
the affine vanishing ideal $J_\bbX$ of $\bbX$ in $K[X_1,...,X_n]$ 
is obtained from $I_\bbX$ by dehomogenization.
The ring $S$ is a 0-dimensional affine $K$-algebra and its zero ideal 
has an irredundant primary decomposition 
$\ideal{0} = \fq_1\cap\cdots\cap\fq_s$.
Then the associated local ring of $\bbX$ at a point
$p_j \in \Supp(\bbX)$ is $\calO_{\bbX,p_j}\cong S/\fq_j$.
Its maximal ideal is denoted by $\fm_{\bbX,p_j}$, and the residue 
field of~$\bbX$ at~$p_j$ is denoted by~$K(p_j)$.
We have $S\cong \prod_{j=1}^s(S/\fq_j)\cong \prod_{j=1}^s\calO_{\bbX,p_j}$
and $\deg(\bbX) = \dim_K(S)=\sum_{j=1}^s \dim_K(\calO_{\bbX,p_j})$.

\begin{definition}
	Let $1\le j\le s$. A subscheme $\bbY\subseteq \bbX$ is called
	a \textit{$p_j$-subscheme} if we have
	$\calO_{\bbY,p_j}\ne \calO_{\bbX,p_j}$ and
	$\calO_{\bbY,p_k} = \calO_{\bbX,p_k}$ for $k \ne j$.
	A $p_j$-subscheme $\bbY\subseteq\bbX$ is called \textit{maximal}
	if $\deg(\bbY) = \deg(\bbX)-\dim_K K(p_j)$.
\end{definition}

When $\bbX$ has $K$-rational support (i.e., all points
$p_1,...,p_s$ are $K$-rational), a maximal $p_j$-subscheme
of~$\bbX$ is nothing but a subscheme $\bbY\subseteq\bbX$
of degree $\deg(\bbY)=\deg(\bbX)-1$ with
$\calO_{\bbY,p_j} \ne \calO_{\bbX,p_j}$.

Given any finitely generated graded $R$-module $M$,
the {\it Hilbert function} of $M$ is a map
$\HF_M: \bbZ \rightarrow \bbN$ given by $\HF_M(i)=\dim_K(M_i)$.
The unique polynomial $\HP_M(z) \in \mathbb{Q}[z]$ for which
$\HF_M(i)=\HP_M(i)$ for all $i\gg 0$ is called the
{\it Hilbert polynomial} of~$M$. The number
$$
\ri(M) = \min\{\, i\in\bbZ\mid \HF_{M}(j)=\HP_M(j)\
\textrm{for all} \ j\geq i \,\}
$$
is called the {\it regularity index} of $M$ (or of $\HF_M$).
Whenever $\HF_M(i)=\HP_M(i)$ for all $i\in\bbZ$,
we let $\ri(M)=-\infty.$

\begin{remark}
We also write $\HF_{\bbX}$ for $\HF_{R}$ 
and call it the Hilbert function of $\bbX$.
Its regularity index is denoted by~$r_{\bbX}.$
We have $\HF_{\bbX}(i)=0$ for $i<0$ and
$$
1=\HF_{\bbX}(0)<\HF_{\bbX}(1)<\cdots<\HF_{\bbX}(r_{\bbX}-1)< \deg(\bbX)
$$
and $\HF_{\bbX}(i)=\deg(\bbX)$ for $i\geq r_{\bbX}$.
\end{remark}

To compare the properties of $\bbX$ and its maximal subschemes, 
one can use sets of separators that are presented below 
(see also \cite{KLL2019}).
Let $Q^h(R)$ be the {\it homogeneous ring of quotients} of~$R$
defined as the localization of~$R$ with respect to the set of
all homogeneous non-zerodivisors of~$R$. By \cite[Section~1]{KLL2019},
there is an injection
\begin{equation}\label{Equa:ImathMap}
\tilde{\imath}: R \rightarrow Q^h(R) \cong
S[x_0,x_0^{-1}] = 
(\prod_{j=1}^s\calO_{\bbX,p_j})[x_0,x_0^{-1}]
\end{equation}
given by $\tilde{\imath}(f) = f^{\rm deh}x_0^i
= (f^{\rm deh}+\fq_1,...,f^{\rm deh}+\fq_s)x_0^i$,
for $f \in R_i$ with $i\ge 0$.
Also, for $i\ge r_\bbX$, the restriction map
$\tilde{\imath}|_{R_i}:R_i \rightarrow (S[x_0,x_0^{-1}])_{i}$
is an isomor\-phism of K-vector spaces. 
Consider a maximal $p_j$-subscheme $\bbY$ of $\bbX$ and 
write $I_{\bbY/\bbX}$ for the vanishing ideal of $\bbY$ in $R$ 
and $R_\bbY = R/I_{\bbY/\bbX}$. Then, by \cite[Proposition~3.2]{KL2017}, 
the ideal $I_{\bbY/\bbX}$ corresponds to one and only one ideal 
$\langle s_j\rangle$ in~$\calO_{\bbX,p_j}$, where $s_j$ is a socle element 
of~$\calO_{\bbX,p_j}$, and $(I_{\bbY/\bbX})^{\rm deh} \subseteq 
\bigcap_{k\ne j}\fq_k$.
For $a \in \calO_{\bbX,p_j}$, we set
$$
\mu(a) := \min\{\, i\in\bbN \,\mid\,
(0,\dots,0,a,0,\dots,0)x_0^i \in \tilde{\imath}(R) \,\},
$$
where $a$ occurs in the $j$-th position of $(0,\dots,0,a,0,\dots,0)$.
Then $\mu(a)\le r_\bbX$ for all $a\in \calO_{\bbX,p_j}$.
Let $\kappa_j := \dim_K K(p_j)$, and let 
$e_{j1}, \dots, e_{j\kappa_j}$ be elements of $\calO_{\bbX,p_j}$ 
whose residue classes form a $K$-basis of $K(p_j)$.
We put 
$$
	f^*_{jk} := \tilde{\imath}^{-1}((0,\dots,0,
	e_{jk}s_{j},0,\dots,0)x_0^{\mu(e_{jk}s_{j})})
$$
and $f_{jk} :=x_0^{r_{\bbX}-\mu(e_{jk}s_{j})}f^*_{jk}$
for $k = 1, \dots, \kappa_j$.

\begin{definition}
\begin{enumerate}
\item[(a)] The set $\{f^*_{j1},\dots,f^*_{j\kappa_j}\}$ is called the
\textit{set of minimal separators of $\bbY$ in $\bbX$}
with respect to $s_j$ and $\{e_{j1},\dots,e_{j\kappa_j}\}$.

\item[(b)] The set $\{f_{j1},\dots,f_{j\kappa_j}\}$ is called the
\textit{set of separators of $\bbY$ in~$\bbX$}
with respect to $s_j$ and $\{e_{j1},\dots,e_{j\kappa_j}\}$.

\item[(c)] The number
$
\mu_{\bbY/\bbX}:=\max\{\, \deg(f^*_{jk})
\mid k = 1, \dots, \kappa_j \,\}
$
is called the {\it maximal degree of a minimal separator
of~$\bbY$ in~$\bbX$}.
\end{enumerate}
\end{definition}

The maximal degree $\mu_{\bbY/\bbX}$ of a minimal separator
of a maximal $p_j$-subscheme $\bbY$ in~$\bbX$ depends neither on the
choice of the socle element~$s_j$ nor on the specific choice
of~$\{e_{j1},\dots,e_{j\kappa_j}\}$ (see \cite[Lemma~4.4]{KL2017}).
Letting $\alpha_{\bbY/\bbX} = \min\{i\in\bbN \mid (I_{\bbY/\bbX})_i\ne 0\}$
the \textit{initial degree} of $I_{\bbY/\bbX}$, 
the Hilbert function of $\bbY$ is described by the following lemma 
(see \cite[Proposition~2.5]{KLL2019}).

\begin{lemma}\label{lem:Sep-MaxSubsch}
For a  maximal $p_j$-subscheme $\bbY$ in~$\bbX$ we have 
$\alpha_{\bbY/\bbX} \le \mu_{\bbY/\bbX} \le r_\bbX$ and
$$
\HF_\bbY(i) = \begin{cases}
\HF_\bbX(i) & \mbox{if $i<\alpha_{\bbY/\bbX}$}\\
\le \HF_\bbX(i)-1 & \mbox{if $\alpha_{\bbY/\bbX}\le i<\mu_{\bbY/\bbX}$}\\
\HF_\bbX(i)-\kappa_j & \mbox{if $i\ge \mu_{\bbY/\bbX}$}
\end{cases}
$$
and there is a set of minimal separators 
$\{f^*_{j1},\dots,f^*_{j\kappa_j}\}$ of $\bbY$ in $\bbX$ such that 
$I_{\bbY/\bbX}=\ideal{f^*_{j1},\dots,f^*_{j\kappa_j}}$.
\end{lemma}

\begin{remark} \label{remSec2.3}
If $\bbX=\{p_1,...,p_s\}$
is a set of $s$ distinct $K$-rational points
of $\mathbb{P}^n_K$, then $\kappa_1=\cdots=\kappa_s=1$.
We write $p_j=(1:a_{j1}:...:a_{jn})$ with $a_{jk}\in K$. 
Then a minimal separator of $\bbX\setminus \{p_j\}$ in $\bbX$ 
is a homogeneous element $f\in R$ of smallest degree 
such that $f(p_j)\ne 0$ and $f(p_k)=0$ if $k\ne j$. 
Let $f_j \in R_{r_\bbX}$ be the separator of $\bbX\setminus \{p_j\}$ 
in $\bbX$ with $f_j(p_j)=1$. 
According to \cite[Corollary~2.9]{KLL2019}, we have 
$R_i =\ideal{ x_0^{i-r_\bbX}f_1,...,x_0^{i-r_\bbX}f_s }_K$
and $g = g(p_1)x_0^{i-r_\bbX}f_1 + \cdots+g(p_s)x_0^{i-r_\bbX}f_s$
for $i\ge r_\bbX$ and $g\in R_i$.
\end{remark}

Now we will introduce the main object of study for this section.
Consider the canonical multiplication map 
$\mu: R\otimes_{K[x_0]}R\rightarrow R$ given by $\mu(f\otimes g)=fg$.
Its kernel $J=\Ker(\mu)$ is a homogeneous ideal of $R\otimes_{K[x_0]}R$.
The graded $R$-algebra $\Omega^1_{R/K[x_0]}=J/J^2$ is known as
the \textit{module of K\"ahler differentials} of $R/K[x_0]$ and 
the homogeneous $K[x_0]$-linear map 
$d_{R/K[x_0]}: R\rightarrow \Omega^1_{R/K[x_0]}$ given by 
$f\mapsto f\otimes1-1\otimes f + J^2$ is the \textit{universal derivation} 
of $R/K[x_0]$.  For any algebra $B/A$ we can define
in the same way the module of K\"ahler differentials $\Omega^1_{B/A}$
(see \cite{Kun1986} for more details). 
In view of \cite[Sections~3-4]{Kun1986}, we deduce the following commutative diagram which indicates the relation between the modules of K\"ahler differentials of the three algebras $R/K[x_0]$, $R/K$ and $S/K$:
$$
\xymatrix{
 & & & 0\ar[d] & \\
 & 0\ar[d] &  & \ideal{x_0-1}\Omega^1_{R/K[x_0]} \ar[d] &  \\
0\ar[r] & Rdx_0 \ar[r]\ar @{^{(}->}[d] & 
\Omega^1_{R/K} \ar[r] \ar @{=}[d]
& \Omega^1_{R/K[x_0]} \ar[r]\ar[d]^{\varepsilon}   & 0\\
0\ar[r] & Rdx_0+ \ideal{x_0-1}\Omega^1_{R/K}  \ar[r] & 
\Omega^1_{R/K} \ar[r]^{\varepsilon} 
& \Omega^1_{S/K} \ar[r]\ar[d]   & 0\\
 & & & 0 & \\
}
$$
where the $R$-linear map $\varepsilon: \Omega^1_{R/K} \rightarrow 
\Omega^1_{S/K}$ is given by $\varepsilon(fdx_0) = 0$ and
$\varepsilon(fdx_i) = f^{\rm deh} dx_i$ for $i=1,...,n$ and $f\in R$.
In addition, $\Omega^1_{\overline{R}/K}\cong 
\Omega^1_{R/K[x_0]}/x_0\Omega^1_{R/K[x_0]}$, 
where $\overline{R}=R/\ideal{x_0}$.
The Fitting ideals of the modules of K\"ahler differentials of
$R/K[x_0]$, $S/K$ and $\overline{R}/K$ are given the following names 
(see \cite[Appendix D]{Kun1986} for basic properties of Fitting ideals). 

\begin{definition}
\begin{enumerate}
\item[(a)] The ideal $\vartheta_\bbX := \vartheta_{R/K[x_0]}
= F_0(\Omega^1_{R/K[x_0]})$ is called the \textit{K\"ahler different} 
of~$\bbX$ (or of $R/K[x_0]$).
\item[(b)] The ideal $\vartheta_\bbX(S/K)= F_0(\Omega^1_{S/K})$ 
is called the \textit{affine K\"ahler different} of~$\bbX$ 
(or of~$S/K$).
\item[(c)] The ideal $\vartheta_\bbX(\overline{R}/K)= 
F_0(\Omega^1_{\overline{R}/K})$  is called the 
\textit{reduced K\"ahler different} of~$\bbX$ (or of~$\overline{R}/K$).
\end{enumerate}
\end{definition}

\begin{remark}\label{remark:3differents}
Given a homogeneous system of generators $\{F_1,...,F_r\}$ 
of~$I_\bbX$, the K\"ahler different $\vartheta_\bbX$ is a homogeneous ideal 
of~$R$ generated by all $n$-minors of the Jacobian matrix 
$\left(\frac{\partial F_j}{\partial x_i}\right)_{\begin{subarray} 
\ i=1,...,n\\ j=1,...,r \end{subarray}}$ 
of $(F_1,...,F_r)$ with respect to $x_1,...,x_n$. 
Also, $\vartheta_\bbX$ does not depend on the choice of a homogeneous 
system of generators of~$I_\bbX$. 
In view of \cite[Proposition~1.1]{GM1984}, one can find a minimal
homogeneous system of generators of $I_\bbX$ of degrees $\le r_\bbX+1$,
and subsequently $\vartheta_\bbX$ can be generated by homogeneous
polynomials of degree $\le nr_\bbX$.
Furthermore, \cite[Rules~10.3]{Kun1986} yields that 
$$
\vartheta_\bbX(S/K) = \vartheta_\bbX/\ideal{x_0-1} = 
\prod_{j=1}^s \vartheta_\bbX(\calO_{\bbX,p_j}/K),\
\vartheta_\bbX(\overline{R}/K) = \vartheta_\bbX/\ideal{x_0}.
$$
\end{remark}

Recall that the local ring $\calO_{\bbX,p_j}=S/\fq_j$ is 
a \textit{complete intersection}, if the preimage $\fQ_j$ 
in $K[X_1,...,X_n]$ of $\fq_j$ can be generated by a regular sequence 
of length $n$ in $K[X_1,...,X_n]$.

\begin{definition}
\begin{enumerate}
\item[(a)] A point $p_j\in\Supp(\bbX)$ is called
a \textit{complete intersection point} of $\bbX$, 
or $\bbX$ is called a \textit{complete intersection at $p_j$}, if 
the local ring $\calO_{\bbX,p_j}$ is a complete intersection. 
\item[(b)] We say that $\bbX$ is a \textit{locally complete intersection} 
if it is a complete intersection at all points of its support.
\item[(c)] We say that $\bbX$ is a \textit{complete intersection} if $I_\bbX$ 
can be generated by a homogeneous regular sequence of length $n$ in~$P$.
\end{enumerate}
\end{definition}

If $\bbX$ is a complete intersection with $I_\bbX=\ideal{F_1,...,F_n}$,  
then it is a locally complete intersection and 
the K\"ahler different $\vartheta_\bbX$ is a principal ideal generated 
by the Jacobian determinant of $(F_1,..,F_n)$.
However, the K\"ahler different can be principal without $\bbX$
being a complete intersection (see \cite[Example~2.13]{KL2017}).

The following characterization of 0-dimensional complete intersections
is collected from a result of Scheja and Storch \cite{SS1975}
and \cite[Proposition~5.4]{KL2017}. Here note that our assumption about the 
characteristic of $K$ is essential for this result to hold.

\begin{proposition}\label{prop:cipoint-ci}
\begin{enumerate}
\item[(a)] The scheme $\bbX$ is a complete intersection at $p_j$ 
if and only if $\vartheta_\bbX(\calO_{\bbX,p_j}/K)\ne 0$. 
In this case, $\vartheta_\bbX(\calO_{\bbX,p_j}/K)$ is the socle 
of $\calO_{\bbX,p_j}$.
\item[(b)] The scheme $\bbX$ is a complete intersection if and only if
$\vartheta_\bbX(\overline{R}/K)\ne 0$.
\end{enumerate}
\end{proposition}

Now we are interested in determining the Hilbert polynomial of 
the K\"ahler different of $\bbX$ and giving upper bound for 
its regularity index. 
If $\bbX$ is a fat point scheme in $\bbP^n_K$, then \cite[Theorem 2.5]{KLL2019}
yields that the Hilbert polynomial of $\vartheta_\bbX$ is exactly 
the number of reduced points of $\bbX$ and its regularity index 
satisfies $\ri(\vartheta_\bbX)= r_\bbX+s-1$ if $n=1$ and 
$\ri(\vartheta_\bbX)\le nr_\bbX$ if $n\ge 2$. 
The next proposition generalizes this property for arbitrary 
$0$-dimensional schemes $\bbX\subseteq \bbP^n_K$.

\begin{proposition}\label{prop_HF-Kdiff}
Let $\bbX_{\rm ci}$ be the set of complete intersection points
in the support $\Supp(\bbX)=\{p_1,...,p_s\}$ of $\bbX$.
Then the Hilbert polynomial of $\vartheta_\bbX$ is given by
$$
\HP_{\vartheta_\bbX} = \sum_{p_j\in \bbX_{\rm ci}} \dim_K K(p_j)
= \sum\limits_{p_j\in \bbX_{\rm ci}} \kappa_j
$$ 
and $\ri(\vartheta_\bbX)\le (n+1)r_\bbX$.
Moreover, we have 
$\ri(\vartheta_\bbX)= r_\bbX+\sum_{p_j\in \bbX_{\rm ci}}\kappa_j-1$ 
if $n=1$ and $\ri(\vartheta_\bbX)\le nr_\bbX$ if $n\ge 2$ and
all points in $\bbX_{\rm ci}$ are reduced.
\end{proposition}
\begin{proof}
Clearly, the map 
$\widetilde{\imath}: R\rightarrow Q^h(R)\cong S[x_0,x_0^{-1}]$
given by (\ref{Equa:ImathMap}) has the image in the subring $S[x_0]$
of $S[x_0,x_0^{-1}]$. For $i\ge \ri(\vartheta_\bbX)$ we have 
$(\vartheta_\bbX)_i 
= x_0^{i-\ri(\vartheta_\bbX)}(\vartheta_\bbX)_{\ri(\vartheta_\bbX)}$,
since $x_0$ is a non-zerodivisor of~$R$. This implies that 
the image of $(\vartheta_\bbX)_{\ri(\vartheta_\bbX)}$ under 
$\widetilde{\imath}$ is $\vartheta_\bbX(S/K)x_0^{\ri(\vartheta_\bbX)}$.
Moreover, the map $\widetilde{\imath}$ is injective, and so we have 
$$
\HP_{\vartheta_\bbX} = \HF_{\vartheta_\bbX}(\ri(\vartheta_\bbX))
=\dim_K(\vartheta_\bbX)_{\ri(\vartheta_\bbX)}
=\dim_K(\vartheta_\bbX(S/K)).
$$ 
According to Remark~\ref{remark:3differents},
$\vartheta_\bbX(S/K)=\prod_{j=1}^s \vartheta_\bbX(\calO_{\bbX,p_j}/K)$, and 
hence 
\begin{equation}\label{Equation:HP_Kdiff}
\HP_{\vartheta_\bbX}=\sum_{j=1}^s\dim_K(\vartheta_\bbX(\calO_{\bbX,p_j}/K)).
\end{equation}
By Proposition~\ref{prop:cipoint-ci}, $\vartheta_\bbX(\calO_{\bbX,p_j}/K)
=0$ if $p_j\notin \bbX_{\rm ci}$, and $\vartheta_\bbX(\calO_{\bbX,p_j}/K)
\ne 0$ if $p_j\in \bbX_{\rm ci}$.
For $p_j\in \bbX_{\rm ci}$, the local ring $\calO_{\bbX,p_j}$ is a 
complete intersection with socle $\vartheta_\bbX(\calO_{\bbX,p_j}/K)$, and
consequently  
$$
\dim_K(\vartheta_\bbX(\calO_{\bbX,p_j}/K)) =
\dim_K(\Ann_{\calO_{\bbX,p_j}}(\fm_{\bbX,p_j}))=\dim_K K(p_j)=\kappa_j.
$$
Based on (\ref{Equation:HP_Kdiff}), we obtain 
the desired equalities for the Hilbert polynomial of~$\vartheta_\bbX$.

Next, we prove the upper bound for $\ri(\vartheta_\bbX)$.
For a field extension $K\subseteq L$, we see that 
$\vartheta_\bbX(L\otimes_KR/L[x_0]) = (L\otimes_KR)\cdot \vartheta_\bbX$
by \cite[Rules~10.3]{Kun1986}. It follows that 
$\HF_{\vartheta_\bbX(L\otimes_KR/L[x_0])} = \HF_{\vartheta_\bbX}$.
So, we may assume that $\bbX$ has $K$-rational support.
Under this assumption, we have $\HP_{\vartheta_\bbX} = \# \bbX_{\rm ci}$.
By renaming the indices, we may write $\bbX_{\rm ci}=\{p_1,...,p_t\}$
with $t\le s$. If $t=0$ then $\vartheta_\bbX=\ideal{0}$, and so there
is nothing to prove. Thus we assume $t\ge 1$.
Let $j\in\{1,...,t\}$. Consider the $K$-linear map 
$$
\widetilde{\imath}_j: R\rightarrow S[x_0] \rightarrow 
\calO_{\bbX,p_j}=S/\fq_j,\, f\mapsto f^{\rm deh}+\fq_j.
$$ 
We have $\widetilde{\imath}_j(\vartheta_\bbX)
= \vartheta_\bbX(\calO_{\bbX,p_j}/K) \ne 0$. 
By Remark~\ref{remark:3differents}, $\vartheta_\bbX$
is generated in degrees $\le nr_\bbX$.
If $\widetilde{\imath}_j(h)=h^{\rm deh}+\fq_j=0$ for all 
$h\in (\vartheta_\bbX)_{nr_\bbX}$, then we have 
$\widetilde{\imath}((\vartheta_\bbX))=\vartheta_\bbX(\calO_{\bbX,p_j}/K)=0$,
a contradiction. Hence there exists $h_j \in (\vartheta_\bbX)_{nr_\bbX}$ 
such that $h_j^{\rm deh}+\fq_j\ne 0$ is a socle element 
in $\calO_{\bbX,p_j}$ for $j=1,...,t$. 
Furthermore, we take $g_j = \widetilde{\imath}^{-1}((0,...,0,1,0,...,0)x_0^{r_\bbX})\in R_{r_\bbX}$, where
$(0,...,0,1,0,...,0)$ has $1$ in the $j$-th position.
Then $f_j := g_jh_j$ is a separator of the unique maximal $p_j$-subscheme
$\bbY_j$ of $\bbX$ of degree $\deg(\bbY_j)=\deg(\bbX)-1$,
since $\calO_{\bbX,p_j}$ is a complete intersection.
In particular, $\{f_1,...,f_t\}$ is $K$-linearly independent and 
$\ideal{f_1,...,f_t}\subseteq (\vartheta_\bbX)_{(n+1)r_\bbX}$,
and this implies 
$$
\HP_{\vartheta_\bbX} = t = \dim_K(\ideal{f_1,...,f_t})_{(n+1)r_\bbX}
\le \HF_{\vartheta_\bbX}((n+1)r_\bbX) \le t.
$$
Thus we get $\HF_{\vartheta_\bbX}((n+1)r_\bbX)=\HP_{\vartheta_\bbX}$, 
and therefore $\ri(\vartheta_\bbX)\le (n+1)r_\bbX$. 

When $n=1$, the scheme $\bbX$ is a fat point scheme, and hence
$\ri(\vartheta_\bbX)= r_\bbX+\#\bbX_{\rm ci} -1$ by \cite[Lemma~2.3]{KLL2019}.
Now we consider the case $n\ge 2$ and all points in $\bbX_{\rm ci}$
are reduced. As the above assumption, $\bbX$ has $K$-rational support, 
and hence each point of $\bbX_{\rm ci}$ is a $K$-rational point of $\bbX$.
For $j\in\{1,...,t\}$, let $f_j^*$ be the minimal separator of 
the uniquely maximal $p_j$-subscheme $\bbY_j$ of $\bbX$, 
and let $F_j\in P$ be the representative 
of $f^*_j$ such that $\deg(F_j) = \deg(f_j^*)\le r_\bbX$.
We want to show that $(f_j^*)^n \in \vartheta_\bbX$. 
Without loss of generality, we may assume that $p_j=(1:0:...:0)$.
Then $X_iF_j \in I_\bbX$ for $i=1,...,n$, whence 
$\frac{\partial(X_1F_j,...,X_nF_j)}{\partial(x_1,...,x_n)}
\in \vartheta_\bbX$. 
We calculate 
$$
\begin{aligned}
\frac{\partial(X_1F_j,...,X_nF_j)}{\partial(X_1,...,X_n)}
&= \det \begin{pmatrix}
F_j + X_1 \frac{\partial F_j}{\partial X_1} & \cdots & 
X_1 \frac{\partial F_j}{\partial X_n}\\
\vdots &\ddots &\vdots \\
X_n \frac{\partial F_j}{\partial X_1} & \cdots & 
F_j + X_n \frac{\partial F_j}{\partial X_n}
\end{pmatrix}\\
&= F_j^n + G,\quad G\in I_\bbX.
\end{aligned}
$$
Hence $(f_j^*)^n \in \vartheta_\bbX$. 
Consequently, we have $\ideal{(f_1^*)^n,...,(f_t^*)^n}
\subseteq \vartheta_\bbX$.
Since each point of $\bbX_{\rm ci}$ is a $K$-rational point, 
$f_j^*(p_j)\ne 0$ and $f_j^*(p_k)=0$ for $k\ne j$, and hence
we conclude that 
$\HF_{\vartheta_\bbX}(nr_\bbX)=
\dim_K(\ideal{(f_1^*)^n,...,(f_t^*)^n})_{nr_\bbX} =
\HP_{\vartheta_\bbX}$ and $\ri(\vartheta_\bbX)\le nr_\bbX$.
\end{proof}

The proposition has the following immediate consequence.

\begin{corollary}\label{cor:LocallyCI}
Under the setting of Proposition~\ref{prop_HF-Kdiff},
suppose that $n\ge 2$.
\begin{enumerate}
\item[(a)] $\bbX$ has a complete intersection point 
if and only if $\vartheta_\bbX \ne 0$.
\item[(b)] $\bbX$ is a locally complete intersection if and only if 
 $\HP_{\vartheta_\bbX}=\sum_{j=1}^s \dim_K K(p_j)$.
\item[(c)] If $\bbX$ is reduced, then $\HP_{\vartheta_\bbX}=\deg(\bbX)$
and $2r_\bbX \le \ri(\vartheta_\bbX)\le nr_\bbX$, in particular,
$\ri(\vartheta_\bbX)=2r_\bbX$ if $n=2$ or $\bbX$ is a complete intersection.
\end{enumerate}
\end{corollary}
\begin{proof}
It remains to prove $2r_\bbX \le \ri(\vartheta_\bbX)$ when $\bbX$
is reduced. By extending the field $K$, we may assume that 
$\bbX$ is a set of $s$ distinct $K$-rational points. Then
the inequality follows from the fact that the K\"ahler different 
is a subideal of the Dedekind different of $\bbX$ (see later)
and from \cite[Corollary~3.6]{KLL2019}.
\end{proof}

The upper bound $\ri(\vartheta_\bbX)\le nr_\bbX$ in
the corollary seems to hold true in general by some experience 
with ApCoCoA, however, we did not know how to prove it.
The hypothesis that $\bbX$ is reduced is essential
for the lower bound $2r_\bbX \le \ri(\vartheta_\bbX)$.

\begin{example}
Consider the non-reduced 0-dimensional complete intersection scheme 
$\bbX\subseteq \bbP^2_\bbQ$ defined by the ideal 
$I_\bbX=\ideal{X_1^2, X_2^3}\subseteq P=\bbQ[X_0,X_1,X_2]$.
Then we have $\HF_\bbX:\ 1\ 3\ 5\ 6\ 6\cdots$ and $r_\bbX=3$.
Moreover, we also have 
$\vartheta_\bbX =\ideal{x_1x_2^2}$ and
$\HF_{\vartheta_\bbX}:\ 0\ 0\ 0\ 1\ 1\cdots$ and 
$\ri(\vartheta_\bbX)=3= r_\bbX$.
\end{example}

\medskip\bigbreak
\section{The K\"ahler Different and CBP for 
Schemes in Generic Position}
\label{Sec3}

In the following we continue to use the assumptions and notation
introduced above. In this section, we want to use the K\"ahler different
to look at a special class of $0$-dimensional schemes $\bbX$ in~$\bbP^n_K$,
namely schemes in generic position. Simultaneously,  we investigate 
the Cayley-Bacharach and arithmetically Gorenstein properties of~$\bbX$. 

We first recall the main notions of this section 
(see also \cite{GO1981} and \cite{KLL2019}).

\begin{definition}
Let $\bbX$ be a 0-dimensional scheme in $\bbP^n_K$ with 
$\Supp(\bbX)\!=\!\{p_1,...,p_s\}$.
\begin{enumerate}
\item[(a)] We say that $\bbX$ is \textit{in generic position}, or $\bbX$ has
\textit{generic Hilbert function}, if
$\HF_{\bbX}(i) = \min\big\{\, \deg(\bbX), \binom{i+n}{n} \,\big\}$
for all $i \in \bbZ$.
\item[(b)] The {\it degree of $p_j$ in $\bbX$} is defined as
$$
\deg_{\bbX}(p_j) := \min\big\{\, \mu_{\bbY/\bbX} \; \big| \;
\bbY \ \textrm{is a maximal $p_j$-subscheme of $\bbX$} \,\big\}.
$$
When $\deg_{\bbX}(p_j)= r_\bbX$ for every $j\in\{1,\dots,s\}$,
we say that $\bbX$ is a \textit{Cayley-Bacharach scheme}
(in short, $\bbX$ is a \textit{CB-scheme}).
\end{enumerate}
\end{definition}

When $n=1$ or ($n \ge 2$ and $\deg(\bbX)=1$), it is clearly true that 
$\bbX$ is a complete intersection in generic position, and so a
CB-scheme, and  that $\vartheta_\bbX$ is a principal ideal 
of~$R$ generated in degree $nr_{\bbX}$.
In the following, we omit this case by assuming that $n,\deg(\bbX) \ge 2$.

Let $\alpha_{\bbX} = \min\{\, i \in \bbN \mid (I_{\bbX})_i \ne 0 \, \}$
be the \textit{initial degree} of~$I_{\bbX}$.
Obviously, $\bbX$ has generic Hilbert function
if and only if $\alpha_\bbX\ge r_\bbX$. In this case, 
$I_{\bbX}$ can be generated by polynomials of degrees
$\alpha_{\bbX}$ and $\alpha_{\bbX}+1$, and $\alpha_{\bbX}$
is the unique integer such that
$$
\binom{n+\alpha_{\bbX}-1}{n} \,\le\, \deg(\bbX) \, < \, 
\binom{n+\alpha_{\bbX}}{n}.
$$
On the other hand, it is well-known that $\bbX$ is
a CB-scheme if and only if, for all $j\in\{1,..,s\}$, every maximal
$p_j$-subscheme $\bbY \subseteq \bbX$ satisfies 
$\dim_K(I_{\bbY/\bbX})_{r_\bbX-1} < \kappa_j=\dim_K K(p_j)$
(see e.g. \cite[Proposition~4.3]{KLL2019}).

We have the following observation.

\begin{lemma} \label{lemSec3.2}
If $\bbX$ is in generic position with 
$\deg(\bbX) = \binom{n+\alpha_{\bbX}-1}{n}$,
then $\bbX$ is a CB-scheme and $\vartheta_\bbX$ is generated 
in degree $nr_\bbX$.
\end{lemma}
\begin{proof}
Observe that $\alpha_\bbX = r_\bbX+1$, and it follows that
$\vartheta_\bbX$ is generated in degree $nr_\bbX$.
The generic position property is clearly invariant under
an extension of base fields, so is the Cayley-Bacharach property 
(see \cite[Corollary~4.9]{KLR2019}). This enables us to assume that
the scheme $\bbX$ has $K$-rational support.
Now let $j\in\{1,...,s\}$ and let $\bbY$ be a maximal 
$p_j$-subscheme of~$\bbX$.
If $F \in (I_\bbY)_{\alpha_{\bbY/\bbX}}$ and
$L \in P_1$ is a linear form through the point $p_j$,
then $FL \in I_{\bbX}$.
This implies $r_\bbX=\alpha_{\bbX}-1 \le \alpha_{\bbY/\bbX} \le r_{\bbX}$.
Thus $\alpha_{\bbY/\bbX} = r_{\bbX}$, especially, 
$\dim_K(I_{\bbY/\bbX})_{r_\bbX-1} =0< 1=\kappa_j$. Therefore 
$\bbX$ is a CB-scheme.
\end{proof}

It is not hard to give an example to show that a $0$-dimensional scheme 
$\bbX \subseteq \bbP^n_K$ in generic position is in general 
not necessary to be a CB-scheme.

Let us consider the subring $\widetilde{R} = S[x_0]
= (\prod_{j=1}^s \calO_{\bbX,p_j})[x_0]$ of $Q^h(R)$.
We have $\widetilde{R}_i\cong R_i$ as $K$-vector spaces
for $i\ge r_\bbX$.  
The \textit{conductor} of $R$ in~$\widetilde{R}$ is the ideal 
$$
\fF_{\widetilde{R}/R}= \{\, f\in Q^h(R) \mid
f\cdot \widetilde{R} \subseteq R \,\}.
$$
When the scheme $\bbX$ is reduced, the ring $\widetilde{R}$ is the 
integral closure of $R$ in its full quotient ring, and hence 
$\fF_{\widetilde{R}/R}$ is the conductor of $R$ in its integral
closure in the traditional sense. Further, $\fF_{\widetilde{R}/R}$
is an ideal of both $R$ and $\widetilde{R}$.  
A characterization of CB-schemes in terms of the conductor 
is given by the following proposition (see \cite[Theorem~5.4]{KLL2019}).

\begin{proposition}\label{prop:CBP-Conductor}
The scheme $\bbX$ is a CB-scheme if and only if 
$\fF_{\widetilde{R}/R} = \bigoplus_{i\geq r_{\bbX}} R_i$.
\end{proposition}

Now we state and prove the converse of Lemma~\ref{lemSec3.2}
for reduced 0-dimensional schemes.

\begin{proposition} \label{prop:GenPos-KDiff-CBP}
Let $\bbX$ be a reduced 0-dimensional scheme in $\bbP^n_K$. 
Then the following conditions are equivalent.
\begin{enumerate}
\item[(a)] $\bbX$ is in generic position with 
$\deg(\bbX)=\binom{n+\alpha_{\bbX}-1}{n}$.
\item[(b)]
The Hilbert function of K\"{a}hler different is given by
$$
\HF_{\vartheta_\bbX}(i) =
		\begin{cases}
		0 & \textrm{if} \ i < nr_{\bbX}, \\
		\deg(\bbX) & \textrm{if} \ i \ge nr_{\bbX}.
		\end{cases}
$$
\item[(c)] $\bbX$ is a CB-scheme and
$\vartheta_\bbX = \fF^n_{\widetilde{R}/R}$.
\end{enumerate}
\end{proposition}
\begin{proof}
We see that the implication ``(a)$\Rightarrow$(b)'' 
follows from Proposition~\ref{prop_HF-Kdiff} and Lemma~\ref{lemSec3.2}.
Now we prove ``(b)$\Rightarrow$(a)''. Suppose that
$\vartheta_\bbX = \ideal{(\vartheta_\bbX)_{nr_\bbX}}$
and $\alpha_{\bbX} \le r_{\bbX}$.
Under an extension of base fields, we may assume that $\bbX$
has $K$-rational support, i.e., that $\bbX=\{p_1,...,p_s\}$ is 
a set of distinct $K$-rational points. 
Due to Remark~\ref{remark:3differents}, let $\{F_1, \dots, F_r\}$
be a minimal homogeneous system of generators of $I_\bbX$ such that
$\deg(F_j) \le r_{\bbX}+1$ for all $j=1,...,r$. In view of the graded
version of the Primbasissatz \cite[Proposition~5.1]{DK1999}, 
we may assume without loss of generality that
$\{F_1, \dots,F_n\}$ is a $P$-regular sequence and the ideal
$\ideal{F_1,...,F_n}$ defines a 0-dimensional complete intersection 
containing $\bbX$ which is reduced at the points of~$\bbX$.
By \cite[Theorem~10.12]{Kun1986}, the Jacobian determinant
$\tfrac{\partial(F_1, \dots,F_n)}{\partial(x_1, \dots,x_n)}$
does not vanish at any point of~$\bbX$.
If there is an index $j \in \{1, \dots, n\}$ such that
$\deg(F_j) \le r_{\bbX}$, then
$\deg(\tfrac{\partial(F_1, \dots, F_n)}{\partial(x_1, \dots, x_n)})
\le nr_{\bbX}-1$. So, we have
$\tfrac{\partial(F_1, \dots, F_n)}{\partial(x_1, \dots, x_n)}
\in \vartheta_\bbX\setminus\{0\}$, and this contradicts
the assumption.

Now we consider the case $\deg(F_j) = r_{\bbX}+1$ for
$j = 1, \dots, n$. Since $\alpha_{\bbX} \le r_{\bbX}$,
there is an element $F_j \in (I_{\bbX})_{\alpha_{\bbX}}$
for some $j \ge n+1$, say $F_{n+1}$.
By the assumption on the characteristic of $K$ and by Euler's rule,
we may assume that $\frac{\partial F_{n+1}}{\partial X_1} \ne 0$.
Clearly, $\frac{\partial F_{n+1}}{\partial X_1} \notin I_{\bbX}$,
as $\deg(\frac{\partial F_{n+1}}{\partial X_1}) < \alpha_{\bbX}$.
So, there is a point~$p_k$ of~$\bbX$ such that
$\frac{\partial F_{n+1}}{\partial x_1}(p_k) \ne 0$.
Without loss of generality, we can assume that
$\frac{\partial F_{n+1}}{\partial x_1}(p_1) \ne 0$.
Set $\calV_i := \big(\frac{\partial F_i}{\partial x_1}(p_1), \dots,
\frac{\partial F_i}{\partial x_n}(p_1)\big)\in K^n$
for $i = 1, \dots, n+1$ and
$$
\calV :=
	\begin{pmatrix}
	\frac{\partial F_1}{\partial x_1}(p_1)& \cdots &
	\frac{\partial F_1}{\partial x_n}(p_1)\\
	\vdots & \ddots &\vdots \\
	\frac{\partial F_n}{\partial x_1}(p_1)& \cdots &
	\frac{\partial F_n}{\partial x_n}(p_1)
	\end{pmatrix}.
$$
Since the matrix $\calV$ is invertible, by using elementary row operations, 
we can transform the matrix~$\calV$ into an upper-triangular 
matrix~$\calW=(w_{ij})$ such that its diagonal entries are all non-zero.
Let $\calW_i$ denote the $i$-th row of the matrix~$\calW$.
Then there are $\lambda_{ij}\in K$, $i,j\in\{1, \dots,n\}$, such that
$$
\begin{aligned}
	\calW_i &=(w_{i1}, \dots,w_{in})
	=\lambda_{i1}\calV_1+\cdots+\lambda_{in}\calV_n \\
	&= \big(\sum_{j=1}^n \lambda_{ij}\frac{\partial F_j}{\partial x_1}(p_1),
	\dots, \sum_{j=1}^n	\lambda_{ij}\frac{\partial F_j}{\partial x_n}(p_1)
	\big).
\end{aligned}
$$
For every $i \in \{2, \dots,n\}$, let
$G_i := \lambda_{i1}F_1+\cdots+\lambda_{in}F_n
\in (I_{\bbX})_{r_{\bbX}+1}\setminus\{0\}$.
Then we have $\frac{\partial G_i}{\partial x_k}(p_1)
=\sum_{j=1}^n \lambda_{ij}\frac{\partial F_j}
{\partial x_k}(p_1)=w_{ik}$
for all $i = 2, \dots, n$ and $k = 1, \dots, n$.
It follows that
$$
\begin{aligned}
\frac{\partial(F_{n+1},G_2, \dots,G_n)}{\partial(x_1, \dots,x_n)}(p_1)
&= \det
\begin{pmatrix}
	\frac{\partial F_{n+1}}{\partial x_1}(p_1)&
	\frac{\partial F_{n+1}}{\partial x_2}(p_1)&
	\cdots & \frac{\partial F_{n+1}}{\partial x_n}(p_1)\\
	0& w_{22} & \cdots &w_{2n}  \\
	\vdots & \vdots & \ddots &\vdots  \\
	0& 0&\cdots & w_{nn}
\end{pmatrix}\\
&=\frac{\partial F_{n+1}}{\partial x_1}(p_1)w_{22}\cdots w_{nn}\ne 0,
\end{aligned}
$$
and hence
$\tfrac{\partial(F_{n+1},G_2, \dots,G_n)}{\partial(x_1, \dots,x_n)}
\in (\vartheta_\bbX)_{\le nr_{\bbX}-1}\setminus\{0\}$,
and this provides a contradiction again.
Thus we must have $\alpha_{\bbX} = r_{\bbX}+1$ and (a) follows.

Next we prove the equivalence of (b) and (c). Suppose that $\bbX$ 
is a CB-scheme and $\vartheta_\bbX=\fF^n_{\widetilde{R}/R}$.
By Proposition~\ref{prop:CBP-Conductor}, we have $\fF_{\widetilde{R}/R}
= \bigoplus_{i \ge r_\bbX} R_i$, and so 
$\vartheta_\bbX = \fF^n_{\widetilde{R}/R}
=\bigoplus_{i \ge nr_\bbX} R_i$.
Thus the Hilbert function of~$\vartheta_\bbX$ is given as claim~(b).
Conversely, if $\vartheta_\bbX = \bigoplus_{i\geq nr_\bbX} R_i$,
then the equivalence of (a) and (b) implies that
$\bbX$ is in generic position with $s=\binom{n+\alpha_{\bbX}-1}{n}$, 
and so $\bbX$ is a CB-scheme by~Lemma~\ref{lemSec3.2}.
Therefore, an application of Proposition~\ref{prop:CBP-Conductor} yields that 
$\fF^n_{\widetilde{R}/R} =\bigoplus_{i\geq nr_{\bbX}} R_i=\vartheta_\bbX$.
\end{proof}

In view of Proposition~\ref{prop:GenPos-KDiff-CBP}, if the Hilbert function
of~$\vartheta_\bbX$ in degree $nr_{\bbX}-1$ is zero
then~$\bbX$ is a CB-scheme. This observation can be improved as follows.

\begin{corollary}\label{cor:HF-KDiff-CBP}
If a reduced 0-dimensional scheme $\bbX\subseteq\bbP^n_K$
has $\HF_{\vartheta_\bbX}(nr_\bbX -n)=0$, then it is a CB-scheme.
\end{corollary}
\begin{proof}
Under an extension of base fields, we may assume that 
$\bbX=\{p_1,...,p_s\}$ is a set of distinct $K$-rational points.
So, the claim follows from \cite[Lemma~3.7]{KLL2015}.
\end{proof}

Unfortunately, the converse of Corollary~\ref{cor:HF-KDiff-CBP} 
fails to hold in general. For instance, a complete intersection set 
of four $K$-rational points in $\bbP^2_K$ is a CB-scheme with
regularity index equal 2, but its K\"ahler different at degree 2
does not vanish.

Next, we consider a reduced 0-dimensional scheme $\bbX\subseteq\bbP^n_K$
in generic position with $\binom{n+\alpha_{\bbX}-1}{n}< \deg(\bbX) 
<\binom{n+\alpha_{\bbX}}{n}$. Let $t$ be the number of minimal generators
of degree~$\alpha_{\bbX}$ of~$I_{\bbX}$. Then $\alpha_\bbX = r_\bbX$
and  $t = \binom{n+\alpha_{\bbX}}{n}-\deg(\bbX)$.
An application of Corollary~\ref{cor:HF-KDiff-CBP} provides the 
following consequence.

\begin{corollary}
If $r_\bbX=1$ or $t < n$, then $\bbX$ is a CB-scheme.
\end{corollary}

Suppose now that $t \ge n$ and $r_\bbX\ge 2$ and that 
$\bbX=\{p_1,...,p_s\}$ is a set of distinct $K$-rational points
(after an extension of base fields).
Write $(I_{\bbX})_{\alpha_\bbX}=\ideal{F_1, \dots, F_t}_K$ and 
$
(\vartheta_\bbX)_{n(r_{\bbX}-1)} \!=\! \ideal{h_1, \dots, h_{\delta}}_K,
$ 
where $h_i = \frac{\partial (F_{i_1}, \dots,F_{i_n})}
{\partial(x_1, \dots,x_n)}$ for 
$\{i_1, \dots, i_n\} \!\subseteq\! \{1, \dots, t\}$ 
and $\delta=\binom{t}{n}$.
For $j \in \{1, \dots, s\}$, we put 
$$
\calA :=
\begin{pmatrix}
h_1(p_1)& h_2(p_1)&\cdots & h_{\delta}(p_1)\\
h_1(p_2)& h_2(p_2)&\cdots & h_{\delta}(p_s)\\
\vdots& \vdots&\ddots & \vdots\\
h_1(p_s)& h_2(p_s)&\cdots & h_{\delta}(p_s)
\end{pmatrix},
\ \, \calA_j :=
\begin{pmatrix}
\calA & \calE_j
\end{pmatrix},
$$
where $\calE_j =(0,...,0,1,0,...,0)^{\rm tr}$ 
with $1$ occurring in the $j$-th position. 

\begin{proposition} \label{propSec4.7}
In the above setting, if $\rank(\calA_j) > \rank(\calA)$ for all 
$j = 1, \dots, s$, then $\bbX$ is a CB-scheme. The converse holds
true for $n=2$ or $(n,r_\bbX)=(3,2)$. 
\end{proposition}
\begin{proof}
Suppose that $\rank(\calA_j) > \rank(\calA)$ for all
$j = 1, \dots, s$, and that $\bbX$ is not a CB-scheme.
For $j\in\{1,...,s\}$, let $f^*_j$ be the minimal separator
of $\bbX\setminus\{p_j\}$ in $\bbX$ with $f^*_j(p_j)=1$
and set $f_j:=x_0^{r_{\bbX}-\deg(f_j^*)}f_j^*$. 
Note that $n(r_{\bbX}-1) \ge r_{\bbX}$. 
For $i \in \{1, \dots, \delta\}$, Remark~\ref{remSec2.3} 
enables us to write 
$$
h_i =  x_0^{n(r_{\bbX}-1)-r_{\bbX}}\sum_{j=1}^s h_i(p_j)f_j.
$$ 
Since $\bbX$ is not a CB-scheme, there exists $j\in\{1,...,s\}$ 
such that $\deg(f^*_j)<r_\bbX$.
Without loss of generality, we may assume that $\deg(f^*_1)<r_\bbX$.
As in the proof of Proposition~\ref{prop_HF-Kdiff}, we have
$
x_0^{n\deg(f_1^*)-r_{\bbX}}f_1 = (f_1^*)^n \in
(\vartheta_\bbX)_{n\deg(f_1^*)},
$
and hence
$x_0^{n(r_{\bbX}-1)-r_{\bbX}}f_1\in(\vartheta_\bbX)_{n(r_{\bbX}-1)}$.
Consequently, there exist $c_1, \dots, c_{\delta} \in K$,
not all equal to zero, such that
$$
\begin{aligned}
	x_0^{n(r_{\bbX}-1)-r_{\bbX}}f_1
	&= c_1h_1+\cdots+c_{\delta}h_{\delta}\\
	&= x_0^{n(r_{\bbX}-1)-r_{\bbX}}
	\big( c_1\sum_{j=1}^s h_1(p_j)f_j + \cdots 
	+c_{\delta}\sum_{j=1}^sh_{\delta}(p_j)f_j
	\big)\\
	&= x_0^{n(r_{\bbX}-1)-r_{\bbX}}
	\big( f_1 \sum_{k=1}^{\delta} h_k(p_1)c_k + \cdots
	+ f_s \sum_{k=1}^{\delta} h_k(p_s)c_k
	\big).
\end{aligned}
$$
This gives a system of linear equations in
$c_1, \dots, c_{\delta}$ as follows
$$
\sum_{k=1}^{\delta} h_k(p_1)c_k=1,\quad 
\sum_{k=1}^{\delta} h_k(p_j)c_k=0,\quad j = 2, \dots, s.
$$
Therefore $\rank(\calA) = \rank(\calA_1)$, in contradiction to
our assumption.

Now let $\bbX$ be a CB-scheme and suppose that
$\rank(\calA_j) = \rank(\calA)$ for some $j \in \{1, \dots, s\}$.
Without loss of generality, we assume that $\rank(\calA_1) = \rank(\calA)$.
Let $\calC = (c_1, \dots, c_{\delta}) \in K^{\delta}$
be a root of the system of linear equations
$\calA \cdot \calY = \calE_1$, where
$\calY = (y_1, \dots, y_{\delta})^{\mathrm{tr}}$.
Then, in~$R$, we have 
$
x_0^{n(r_{\bbX}-1)-r_{\bbX}}f_1 -
\sum_{k=1}^{\delta} c_kh_k =0,
$
and so $x_0^{n(r_{\bbX}-1)-r_{\bbX}} f_1\in (\vartheta_\bbX)_{n(r_{\bbX}-1)}$.
By the assumption, $n(r_{\bbX}-1) \le 2r_{\bbX}-1$, and hence
$x_0^{r_{\bbX}-1} f_1 \in (\vartheta_\bbX)_{2r_{\bbX}-1}$.
By \cite[Corollary~4.6]{KLL2019}, $\bbX$ is not a CB-scheme, 
a contradiction.
\end{proof}

We obtain the following immediate consequence.

\begin{corollary} \label{corSec4.11}
Let $\bbX\subseteq\bbP^n_K$ be a set of $s$ distinct $K$-rational points
in generic position with $s = \binom{n+\alpha_{\bbX}}{n}-n$.
If $(I_{\bbX})_{\alpha_{\bbX}}$ contains a $P$-regular sequence 
of length~$n$ which meet in precisely $\alpha_{\bbX}^n$ distinct 
$K$-rational points, then~$\bbX$ is a CB-scheme.
\end{corollary}

\begin{lemma}
Let $\bbX\subseteq \bbP^n_K$ be a 0-dimensional scheme 
with $\tbinom{n+(\alpha-1)}{n} < \deg(\bbX) < \tbinom{n+\alpha}{n}$
for some $\alpha> 1$.
If $\HF_{\bbX}(\alpha) = \deg(\bbX)$
and $\dim_K(I_{\bbX})_{\alpha} \le n$, then $\bbX$ is in generic position
with $\deg(\bbX) = \binom{n+\alpha}{n} - \dim_K(I_{\bbX})_{\alpha}$.
\end{lemma}
\begin{proof}
Observe that $r_{\bbX} \le \alpha$ and
$\HF_{\bbX}(\alpha-1) \le \binom{n+(\alpha-1)}{n} < \deg(\bbX)$.
This implies $r_{\bbX} = \alpha$.
Suppose that $\alpha_{\bbX} < r_{\bbX}$.
Let $F$ be a non-zero form of degree $r_{\bbX}-1$
in $(I_{\bbX})_{r_{\bbX}-1}$.
Then $X_0F, X_1F,  \dots, X_nF$ are $K$-linearly independent forms
of degree~$r_{\bbX}$ in~$(I_{\bbX})_{r_{\bbX}}$. Thus
$\dim_K(I_{\bbX})_{r_{\bbX}} = \dim_K(I_{\bbX})_{\alpha} > n$,
a contradiction. Hence we conclude that $\alpha_{\bbX} \ge r_{\bbX}$,
and therefore $\bbX$ is in generic position.
\end{proof}

\begin{proposition} \label{propSec4.10}
Let $\bbX\subseteq \bbP^n_K$ be a reduced 0-dimensional scheme 
such that $\tbinom{n+(\alpha-1)}{n}< \deg(\bbX) < \tbinom{n+\alpha}{n}$
for some $\alpha> 1$.
Assume that the following conditions are satisfied:
\begin{enumerate}
	\item[(i)] $\HF_{\bbX}(\alpha) = \deg(\bbX)$,

	\item[(ii)] $\dim_K(I_{\bbX})_{\alpha} \ge n$,

	\item[(iii)]
	there are $n$ independent forms in $(I_{\bbX})_{\alpha}$
	which meet in precisely $\alpha^n$ distinct $L$-rational points
	in $\bbP^n_L$ for a field extension $K\subseteq L$.
\end{enumerate}
Then $\bbX$ is in generic position if and only if
$\HF_{\vartheta_\bbX}(n\alpha - n - 1) = 0$.
\end{proposition}
\begin{proof}
First we notice that, by \cite[Proposition~4.4]{GM1984},
for $s= \deg(\bbX)$ there is a non-empty open set in~$(\bbP^n_L)^s$
whose each point corresponds to a set of $s$ distinct
$L$-rational points satisfies conditions (ii)-(iii).
Secondly, both the generic position property and the Hilbert function
of $\vartheta_\bbX$ are invariant under the extension $K\subseteq L$,
we may assume that $L=K$ and $\bbX$ is a set of $K$-rational points.

Now suppose that $\bbX$ is in generic position. 
Then the condition (i) implies $r_{\bbX}=\alpha=\alpha_{\bbX}$.
Thus $\vartheta_\bbX$ is generated by homogeneous
elements of~degree $\ge n\alpha-n$,
and subsequently $\HF_{\vartheta_\bbX}(n\alpha-n-1)=0$.

Conversely, suppose that $\HF_{\vartheta_\bbX}(n\alpha-n-1)=0$.
Based on the conditions~(ii)-(iii), we let
$F_1,\dots, F_n \in (I_{\bbX})_{\alpha}$
be $n$ forms of degree~$\alpha$ having exactly $\alpha^n$
distinct common zeros in~$\bbP^n_K$.
Then $F_1, \dots,F_n$ form a $P$-regular sequence
by~\cite[Theorem~1.12]{KK1987}, and so the ideal
$\langle F_1, \dots,F_n \rangle$ defines
a complete intersection set $\bbW$ of $\alpha^n$ distinct
$K$-rational points and $\bbX\subseteq \bbW$. 
By \cite[Lemma~3.1]{KLL2015}, the homogeneous element
$\frac{\partial(F_1, \dots,F_n)}{\partial(x_1, \dots,x_n)}$
is a non-zerodivisor of~$P/\langle F_1, \dots,F_n \rangle$
of degree $n\alpha-n$. In particular,
it does not vanish at any point of~$\bbW$, and hence of~$\bbX$.
If $\alpha_{\bbX} < \alpha$, then we can argue analogously
as in the proof of Proposition~\ref{prop:GenPos-KDiff-CBP} 
to obtain a non-zero homogeneous element of degree 
$n\alpha-n-1$ in~$\vartheta_\bbX$.
This implies $\HF_{\vartheta_\bbX}(n\alpha-n-1) \ne 0$,
in contradiction to the assumption.
Therefore we must have $\alpha_{\bbX} = \alpha = r_{\bbX}$.
In other words, $\bbX$ is in generic position, as wanted.
\end{proof}

Next, we look at the Gorenstein property of a $0$-dimensional 
scheme $\bbX \subseteq \bbP^n_K$. Recall that $\bbX$ is  
\textit{locally Gorenstein} if the local ring $\mathcal{O}_{\bbX,p_j}$
is a local Gorenstein ring for $j=1,...,s$; and it is
\textit{arithmetically Gorenstein} if $R$ is a Gorenstein ring.
It is well-known that every (locally) complete intersection is  
(locally) arithmetically Gorenstein, and that 
every arithmetically Gorenstein scheme $\bbX$ is also
a locally Gorenstein CB-scheme (see \cite{KL2017}).

The number $\ell(R/\fF_{\widetilde{R}/R})$ is known as the
\textit{conductor colength} of~$\bbX$, 
where ``$\ell$'' denotes length (or dimension) as $K$-vector space.
The lengths $\ell(R/\fF_{\widetilde{R}/R})$ and $\ell(\widetilde{R}/R)$ 
are finite, since $f\cdot\widetilde{R}\subseteq R$ for all $f\in R_i$ 
and $i\ge r_\bbX$.
In our setting, the Ap\'{e}ry-Gorenstein-Samuel theorem 
(cf. \cite[Corollary~6.5]{Bas1963} or \cite[Section~3]{GO1981}) 
states that if $\bbX$ is reduced, then it is arithmetically Gorenstein if 
and only if $\ell(\widetilde{R}/R)=\ell(R/\mathfrak{F}_{\widetilde{R}/R})$.
When $\bbX$ is in generic position (not necessarily reduced),
we prove the following version of this theorem.

\begin{thm}[Ap\'{e}ry-Gorenstein-Samuel]\label{thm:Apery-Gorenstein-Samual}
Let $\bbX\subseteq \bbP^n_K$ be a $0$-dimensional locally Gorenstein 
scheme in generic position.
Then $\bbX$ is arithmetically Gorenstein if and only if 
$\ell(\widetilde{R}/R) = \ell(R/\mathfrak{F}_{\widetilde{R}/R})$.
\end{thm}

In the proof of this theorem we use the following lemma.

\begin{lemma}\label{lemSec3.3}
If $\bbX$ is in generic position, then we have 
\begin{equation}\label{Equation3.4}
r_\bbX\cdot \deg(\bbX) \ge 2\sum_{i=0}^{r_\bbX-1}\HF_\bbX(i).
\end{equation}
\end{lemma}
\begin{proof}
If $n=1$ then $\bbX$ is a complete intersection with 
$r_\bbX = \deg(\bbX)-1$ and 
$\sum_{i=0}^{r_\bbX-1}\HF_\bbX(i) = \frac{(\deg(\bbX)-1)\deg(\bbX)}{2}$,
and so the equality for (\ref{Equation3.4}). 
Also, the inequality (\ref{Equation3.4}) is trivial for 
the case $r_\bbX=0$. Thus we may assume that $n\ge 2$
and $r_\bbX\ge 1$, and consider the following two cases:

\noindent \textit{Case 1:}\quad Suppose 
$\deg(\bbX)= \binom{n+\alpha_{\bbX}-1}{n}$.
Then $r_\bbX=\alpha_\bbX-1$ and 
$$
\sum_{i=0}^{r_\bbX-1}\HF_\bbX(i) = \sum_{i=0}^{r_\bbX-1}\binom{n+i}{n}
= \binom{n+\alpha_\bbX-1}{n+1}.
$$ 
So, the inequality (\ref{Equation3.4}) is equivalent to
$$
r_\bbX\binom{n+r_\bbX}{n} \ge 2\binom{n+r_\bbX}{n+1}.
$$
But the last inequality follows from the fact that 
$$
r_\bbX\binom{n+r_\bbX}{n} -2\binom{n+r_\bbX}{n+1}
= \frac{(n+r_\bbX)!}{n!(r_\bbX-1)!}(1-\frac{2}{n+1})\ge 0.
$$ 

\noindent \textit{Case 2:}\quad Suppose 
$\binom{n+\alpha_{\bbX}-1}{n}< \deg(\bbX)< \binom{n+\alpha_{\bbX}}{n}$.
Then $r_\bbX = \alpha_\bbX$ and 
$$
\sum_{i=0}^{r_\bbX-1}\HF_\bbX(i) = \sum_{i=0}^{r_\bbX-1}\binom{n+i}{n}
= \binom{n+\alpha_\bbX}{n+1}.
$$
For $\alpha_\bbX=1$ we have $r_\bbX\cdot \deg(\bbX)=\deg(\bbX)\ge 2 
= 2\binom{n+\alpha_\bbX}{n+1}$. 
In order to prove the inequality (\ref{Equation3.4}),
it suffices to prove the inequality 
\begin{equation}\label{Equation3.4-2}
\alpha_\bbX(\binom{n+\alpha_{\bbX}-1}{n}+1) \ge 2\binom{n+\alpha_\bbX}{n+1}
\end{equation}
for all $\alpha_\bbX\ge 2$ and $n\ge 2$. 
When $\alpha_\bbX = 2$, we see that 
$\alpha_\bbX(\binom{n+\alpha_{\bbX}-1}{n}+1)
= 2(\binom{n+1}{n}+1) = 2n+4 = 2\binom{n+\alpha_\bbX}{n+1}$.
Now we assume that the inequality (\ref{Equation3.4-2}) holds true
for some $\alpha_\bbX\ge 2$. We need to prove that the inequality
(\ref{Equation3.4-2}) also holds true for $\alpha_\bbX+1$.
Observe that 
$$
\begin{aligned}
&(\alpha_\bbX+1)(\binom{n+\alpha_{\bbX}}{n}+1) - 
2\binom{n+\alpha_\bbX+1}{n+1} \\
&= \alpha_\bbX(\binom{n+\alpha_{\bbX}}{n}+1) + \binom{n+\alpha_{\bbX}}{n}+1 
- 2(\binom{n+\alpha_\bbX}{n+1}+\binom{n+\alpha_\bbX}{n})\\
&\ge \alpha_\bbX\binom{n+\alpha_{\bbX}-1}{n-1} 
-\binom{n+\alpha_\bbX}{n} +1\\
&= \binom{n+\alpha_\bbX}{n}(\frac{n\alpha_\bbX}{n+\alpha_\bbX} -1)+1 \ge 0
\end{aligned} 
$$
since $n\alpha_\bbX\ge n+\alpha_\bbX$ for $n,\alpha_\bbX\ge 2$.
Hence the the inequality (\ref{Equation3.4-2}) holds true
for all $\alpha_\bbX\ge 2$ and $n\ge 2$, as desired.
\end{proof}

\begin{proof}[Proof of Theorem~\ref{thm:Apery-Gorenstein-Samual}]
Suppose that $\bbX$ is arithmetically Gorenstein. Then
$\bbX$ is a CB-scheme and $\HF_\bbX$ is symmetric 
by \cite[Theorem~6.8]{KLR2019}. So, Proposition~\ref{prop:CBP-Conductor}
yields that $\fF_{\widetilde{R}/R} = \bigoplus_{i\geq r_{\bbX}} R_i$.
We have the isomorphism
$\widetilde{\imath}|_{R_{r_\bbX+i}}: R_{r_\bbX+i}
\rightarrow \widetilde{R}_{r_\bbX+i}$ of $K$-vector spaces
for every $i\ge 0$.
Since $\HF_\bbX(i)=\deg(\bbX)-\HF_\bbX(r_\bbX-1 -i)$ for all $i\in\bbZ$, 
this implies  
$$
\ell(\widetilde{R}/R) = \sum_{i=0}^{r_\bbX-1}(\deg(\bbX)-\HF_\bbX(i))
= \sum_{i=0}^{r_\bbX-1}\HF_\bbX(i)= \ell(R/\mathfrak{F}_{\widetilde{R}/R}).
$$
Conversely, suppose that
$\ell(\widetilde{R}/R) = \ell(R/\mathfrak{F}_{\widetilde{R}/R})$.
We write $\mathfrak{F}_{\widetilde{R}/R}=\mathfrak{F}_1\times \cdots \times
\mathfrak{F}_s$ as a direct product of ideals $\mathfrak{F}_j$ in
$\mathcal{O}_{\bbX,p_j}[x_0]$.
Using the notation $\mu(a) =\min\{ i\in\bbN \mid (0,...,0,a,0,...,0)x_0^i\in\widetilde{\imath}(R)\}$ and 
$\nu(a)=\max\{ \mu(ab) \mid 0\ne b\in \mathcal{O}_{\bbX,p_j}\}$ 
for $a\in \mathcal{O}_{\bbX,p_j}$, \cite[Proposition~5.2]{KLL2019} yields that 
$$
\mathfrak{F}_j = \langle  ax_0^{\nu(a)} \mid a \in \mathcal{O}_{\bbX,p_j}
\setminus\{0\} \rangle.
$$
Note that for a non-zero element $a\in \mathcal{O}_{\bbX,p_j}$ 
there is $b\in \mathcal{O}_{\bbX,p_j}$ such that $ab$ is a socle 
element of $\mathcal{O}_{\bbX,p_j}$, and so 
$\deg_\bbX(p_j)\le \nu(a)\le r_\bbX$.
In particular, there is a socle element $s_j\in\mathcal{O}_{\bbX,p_j}$
such that $\nu(s_j)=\deg_\bbX(p_j)$ by the definition of the degree 
of a point in~$\bbX$. Let $m_j = \dim_K(\mathcal{O}_{\bbX,p_j})$
for $j=1,...,s$. We see that 
$$
m_j\deg_\bbX(p_j) \le \dim_K(\mathcal{O}_{\bbX,p_j}[x_0]/\mathfrak{F}_j) 
\le m_j r_\bbX
$$
and the equalities only occur if $\deg_\bbX(p_j)=r_\bbX$.
It follows that
$$
\begin{aligned}
\sum_{j=1}^s m_j\deg_\bbX(p_j) 
\le \ell(\widetilde{R}/\mathfrak{F}_{\widetilde{R}/R})
= \sum_{j=1}^s \dim_K(\mathcal{O}_{\bbX,p_j}[x_0]/\mathfrak{F}_j)
\le \sum_{j=1}^s m_jr_\bbX
\end{aligned}
$$
and with equalities if and only if $\deg_\bbX(p_j)=r_\bbX$ for all $j=1,...,s$.
Moreover, we have $\deg(\bbX)=\sum_{j=1}^s m_j$ and
$$
\ell(\widetilde{R}/R) = \sum_{i=0}^{r_\bbX-1} 
(\deg(\bbX)-\HF_\bbX(i)) = r_\bbX\deg(\bbX) - 
\sum_{i=0}^{r_\bbX-1}\HF_\bbX(i).
$$
From the exact sequence 
$$
0\longrightarrow R/\mathfrak{F}_{\widetilde{R}/R}
\stackrel{\widetilde{\imath}}{\longrightarrow} 
\widetilde{R}/\mathfrak{F}_{\widetilde{R}/R} 
\longrightarrow \widetilde{R}/R 
\longrightarrow 0
$$
we also have 
$$
\ell(R/\mathfrak{F}_{\widetilde{R}/R}) = 
\ell(\widetilde{R}/\mathfrak{F}_{\widetilde{R}/R})-\ell(\widetilde{R}/R).
$$
So, the equality 
$\ell(\widetilde{R}/R) = \ell(R/\mathfrak{F}_{\widetilde{R}/R})$
implies 
$$
\begin{aligned}
r_\bbX\deg(\bbX) - 
\sum_{i=0}^{r_\bbX-1}\HF_\bbX(i)
&= \sum_{j=1}^s \dim_K(\mathcal{O}_{\bbX,p_j}[x_0]/\mathfrak{F}_j)
- r_\bbX\deg(\bbX) + 
\sum_{i=0}^{r_\bbX-1}\HF_\bbX(i) \\
&\le  \sum_{i=0}^{r_\bbX-1}\HF_\bbX(i).
\end{aligned}
$$
On the other hand, by Lemma~\ref{lemSec3.3}, we have 
$r_\bbX\deg(\bbX) \ge 2\sum_{i=0}^{r_\bbX-1}\HF_\bbX(i)$.
Thus we must have $r_\bbX\deg(\bbX) = 2\sum_{i=0}^{r_\bbX-1}\HF_\bbX(i)$.
This means  
$\sum_{j=1}^s \dim_K(\mathcal{O}_{\bbX,p_j}[x_0]/\mathfrak{F}_j)
= r_\bbX\deg(\bbX)$, i.e., $\deg_\bbX(p_j)=r_\bbX$ for all
$j=1,...,s$. Hence $\bbX$ is a CB-scheme.
Let $\omega_R=\Hom_{K[x_0]}(R,K[x_0])(-1)$ be the canonical module of $R$.
By \cite[Proposition~4.3]{KLLT2020}, we have 
$\Ann_R((\omega_R)_{-r_\bbX+1}) = \ideal{0}$. 
Because $\bbX$ is locally Gorenstein, there is for each $j\in\{1,...,s\}$
a uniquely maximal $p_j$-subscheme $\bbY_j$ of $\bbX$. So, 
$\Ann_R((\omega_R)_{-r_\bbX+1}) = \ideal{0}$ is equivalent to the
condition that for each $j\in\{1,...,s\}$ there is 
$\varphi_j\in(\omega_R)_{-r_\bbX+1}$ such that 
$I_{\bbY_j/\bbX}\cdot \varphi_j \ne \ideal{0}$.
Since $I_{\bbY_j/\bbX}$ can be generated by a set of minimal separators 
of $\bbY_j$ in $\bbX$ by Lemma~\ref{lem:Sep-MaxSubsch}, 
we can choose a separator 
$f_j\in (I_{\bbY_j/\bbX})_{r_\bbX}$ and $g_j\in R_{k_j}$ such that 
$f_j\cdot\varphi_j (g_j) = \varphi_j(f_jg_j)=a_{jj}x_0^{k_j}$ 
with $a_{jj}\in K\setminus\{0\}$. 
We want to find $\varphi = c_1\varphi_1+\cdots+c_s\varphi_s$ with 
$c_1,...,c_s\in K$ such that $f_j\cdot \varphi \ne 0$ for $j=1,...,s$,
and hence $\Ann_R(\varphi)=\ideal{0}$ by \cite[Lemma~4.2]{KLLT2020}.
Letting $\varphi_i(f_jg_j)=a_{ij}x_0^{k_j}\in K[x_0]$ with
$a_{ij}\in K$, we have 
$$
f_j\cdot\varphi(g_j)= c_1\varphi_1(f_jg_j)+\cdots+c_s\varphi_s(f_jg_j)
= (c_1a_{1j}+\cdots+c_sa_{sj})x_0^{k_j} \ne 0
$$
if and only if $c_1a_{1j}+\cdots+c_sa_{sj}\ne 0$. 
Consider the non-constant polynomial 
$h = \prod_{i=1}^s(z_1a_{1j}+\cdots+z_sa_{sj})$ 
in indeterminates $z_1,...,z_s$ (as $a_{jj}\ne 0$). Since 
$\#K>\deg(\bbX)\ge s$, we can argue as in \cite[Proposition~5.5.21]{KR2005} 
to find a tuple $(c_1,...,c_s)\in K^s$ that is not a zero of $h$.
For such tuple $(c_1,...,c_s)\in K^s$, the element 
$\varphi\in (\omega_R)_{-r_\bbX+1}$ satisfies
$\Ann_R(\varphi)=\ideal{0}$.

Consequently, the homogeneous $R$-linear map 
$R(r_\bbX-1)\rightarrow \omega_R,1\mapsto \varphi,$ 
is an injection, and hence we get
$$
\HF_\bbX(i) \le \HF_{\omega_R}(-r_\bbX+1+i)=\deg(\bbX)-\HF_\bbX(r_\bbX-1-i)
$$
for all $i\in\bbZ$. Based on the equality $\sum_{i=0}^{r_\bbX-1} 
(\deg(\bbX)-\HF_\bbX(i)) = \sum_{i=0}^{r_\bbX-1}\HF_\bbX(i)$,
the Hilbert function $\HF_\bbX$ is symmetric. 
An application of \cite[Theorem~6.8]{KLR2019} yields that 
$\bbX$ is arithmetically Gorenstein.   
\end{proof}

\begin{remark}
In general, if $\bbX$ is arithmetically Gorenstein 
(may not in generic position) then we always have
$\ell(\widetilde{R}/R) = \ell(R/\mathfrak{F}_{\widetilde{R}/R})$.
\end{remark}

\begin{example}
Let $\bbX\subseteq \bbP^3_K$ be the 0-dimensional scheme 
consisting of 3 $K$-rational points $p_1=(1:0:0:0)$, $p_2=(1:1:0:0)$,
$p_3=(1:0:1:0)$ and a non-reduced point $p_4$ corresponding to
$I_{p_4} = \langle X_1-X_0, X_2-X_0, (X_3-X_0)^2 \rangle 
\subseteq K[x_0,...,X_3]$. Then $\deg(\bbX)=5$ and 
$\HF_\bbX: 1\ 4\ 5\ 5\cdots$, and so $\alpha_\bbX = r_\bbX=2$.
Thus $\bbX$ is in generic position with 
$r_\bbX\deg(\bbX) = 2\sum_{i=0}^{r_\bbX-1}\HF_\bbX(i)$. 
Also, we have $\deg_\bbX(p_i)=2$ for $i=1,...,4$, and hence 
$\ell(R/\mathfrak{F}_{\widetilde{R}/R}) = 
\ell(\widetilde{R}/\mathfrak{F}_{\widetilde{R}/R}) -\ell(\widetilde{R}/R)
= 2(1+1+1+2) - 5 = 5 =\ell(\widetilde{R}/R)$.
Therefore $\bbX$ is an arithmetically Gorenstein scheme.

Next, consider the 0-dimensional scheme $\mathbb{Y} \subseteq \bbP^3_K$
of degree 5 with the same support and $K$-rational points as $\bbX$, 
but the non-reduced point $p_4$ of~$\mathbb{Y}$ corresponds to
$I'_{p_4} = \langle (X_1-X_0)^2, X_2-X_0, X_3-X_0 \rangle$. 
Then $\HF_\bbY: 1\ 4\ 5\ 5\cdots$ and $r_\bbY=\alpha_\bbY=2$,
and so $\bbY$ is in generic position. However, we see that
$X_2-X_3$ is the minimal separator of $\bbY\setminus\{p_3\}$ in $\bbY$, 
and hence $\deg_\bbY(p_3)=1$, while $\deg_\bbY(p_1)=\deg_\bbY(p_2)=2$.
Moreover, $\bbY$ has a uniquely maximal $p_4$-subscheme 
whose minimal separator is of degree 2, and hence $\deg_\bbY(p_4)=2$. 
This shows that 
$\ell(\widetilde{R}_\bbY/\mathfrak{F}_{\widetilde{R}_\bbY/R_\bbY}) 
= 2+2+1+4=9$ and $\ell(\widetilde{R}_\bbY/R_\bbY)=5 \ne 4 
=\ell(R_\bbY/\mathfrak{F}_{\widetilde{R}_\bbY/R_\bbY})$.
Therefore the scheme $\bbY$ is not arithmetically Gorenstein. 
In this case we also have 
$r_\bbY\deg(\bbY) = 2\sum_{i=0}^{r_\bbY-1}\HF_\bbY(i)$.
\end{example}

\begin{corollary}\label{corSec3.5-1}
Let $\bbX$ be a $0$-dimensional locally Gorenstein scheme in~$\bbP^n_K$.
Then $\bbX$ is arithmetically Gorenstein if and only if 
it is a CB-scheme and 
$\ell(\widetilde{R}/R) = \ell(R/\mathfrak{F}_{\widetilde{R}/R})$.
\end{corollary}
\begin{proof}
If $\bbX$ is a CB-scheme, then 
$\HF_\bbX(i)\le \deg(\bbX)-\HF_\bbX(r_\bbX-1-i)$ for all $i\in\bbZ$
and $\mathfrak{F}_{\widetilde{R}/R}=\oplus_{i\ge r_\bbX}R_i$, and so 
$\ell(R/\mathfrak{F}_{\widetilde{R}/R})=\sum_{i=0}^{r_\bbX-1}\HF_\bbX(i)$
and $\ell(\widetilde{R}/R) = \sum_{i=0}^{r_\bbX-1}(\deg(\bbX)-\HF_\bbX(i))$.
Thus $\ell(R/\mathfrak{F}_{\widetilde{R}/R})=\ell(\widetilde{R}/R)$ is equivalent to the condition that $\HF_\bbX$ is symmetric, 
in turn equivalent to that $\bbX$ is arithmetically Gorenstein 
by \cite[Theorem~6.8]{KLR2019}.
\end{proof}

\medskip\bigbreak
\section{K\"ahler Differents of Complete Intersections}
\label{Sec4}

As in the previous sections, we let $\bbX \subseteq \bbP^n_K$
be a $0$-dimensional scheme over a field $K$ of characteristic zero
or ${\rm char}(K)>\deg(\bbX)$, and let $\Supp(\bbX)=\{p_1,...,p_s\}$. 
The aim of this section is to examine the relation between
the K\"ahler different and two other differents, namely the Noether 
and Dedekind differents, and use them to characterize 
arithmetically Gorenstein schemes and complete intersections.

First we describe the K\"ahler different $\vartheta_\bbX$
when $\bbX$ is a complete intersection.
For $n=1$, $\bbX$ is a complete intersection and $\vartheta_\bbX$  
is a principal ideal whose Hilbert function is of form
$$
\HF_{\vartheta_\bbX}:\ 0 \ \cdots \ 0 \ \underset{\tiny{[d-1]}}{1}
   \ 2 \ 3 \cdots \ t-1 \ \underset{\tiny{[d+t-2]}}{t} \ t \cdots.
$$
and $\ri(\vartheta_\bbX) = d+t-2$, where 
$t= \sum_{j=1}^s\dim_KK(p_j)$ and $d=\deg(\bbX)$.
The straightforward, but tedious proof of this property 
is left to the interested reader.

In the following we consider the case $n\ge 2$. When $\bbX$ is a 
reduced complete intersection, \cite[Proposition~2.6]{KL2017}
provides us a concrete description of the Hilbert function 
of $\vartheta_\bbX$. More generally, we supplement the following lemma.

\begin{lemma}
Let $\bbX$ be a 0-dimensional complete intersection in $\bbP^n_K$,
and let $I_\bbX=\ideal{F_1,...,F_n}$, where $F_j\in P$
is a homogeneous polynomial of degree $d_j$ for $j=1,...,n$.
\begin{enumerate}
\item[(a)] The K\"ahler different of $\bbX$ is given by $\vartheta_\bbX = 
\ideal{\frac{\partial(F_1,...,F_n)}{\partial(x_1,...,x_n)}}$.
\item[(b)] The Hilbert polynomial of $\vartheta_\bbX$ is 
$\HP_{\vartheta_\bbX}= \sum_{j=1}^s\dim_KK(p_j)$ and its regularity 
index satisfies 
$r_\bbX=\sum_{i=1}^nd_i-n \le \ri(\vartheta_\bbX)\le 2r_\bbX$. 
In particular, the lower bound for $\ri(\vartheta_\bbX)$ is attained 
if the support of $\bbX$ contains only one $K$-rational point, 
while the upper bound  is attained when $\bbX$ is reduced.
\end{enumerate}
\end{lemma}
\begin{proof}
This follows from the definition of $\vartheta_\bbX$ and
by Proposition~\ref{prop_HF-Kdiff},
especially, the upper bound for $\ri(\vartheta_\bbX)$ follows from
the fact that the injection $\widetilde{\imath}: R\rightarrow S[x_0]$
given by (\ref{Equa:ImathMap}) sends 
$\frac{\partial(F_1,...,F_n)}{\partial(x_1,...,x_n)}\cdot R_{r_\bbX}$ 
onto $\vartheta_\bbX(S/K)x_0^{2r_\bbX}$.
\end{proof}

The following example shows that the case 
$r_\bbX < \ri(\vartheta_\bbX) < 2r_\bbX$ can occur
for a 0-dimensional complete intersection $\bbX\subseteq \bbP^n_K$.

\begin{example}
Let $\bbX$ be the non-reduced 0-dimensional complete intersection 
of two cubic curves in $\bbP^2_K$ defined by $I_\bbX=\ideal{F,G}$,
where $F = X_0 X_1^2  +X_1^3$ and 
$G = X_0^2 X_2 -2 X_0X_2^2  +X_2^3$.
Then $\deg(\bbX)=9$ and the support of $\bbX$ contains 
four $K$-rational points. We have $\HF_\bbX: \ 1\ 3\ 6\ 8\ 9\ 9\cdots$ 
and $r_\bbX=4$. Furthermore, the K\"ahler different of $\bbX$ is 
$\vartheta_\bbX=\ideal{ 2 x_0^3 x_{1} +3 x_{1}^4  +12 x_{1}^3 x_{2} 
-10 x_0 x_{1} x_{2}^2  +9 x_{1}^2 x_{2}^2  +8 x_{1} x_{2}^3}$
and $\HF_{\vartheta_\bbX}:\ 0\ 0\ 0\ 0\ 1\ 3\ 4\ 4\cdots$.
In particular, $r_\bbX=4< \ri(\vartheta_\bbX)=6 <8=2r_\bbX$.
\end{example}

Now we introduce the Noether different of a 0-dimensional scheme 
$\bbX\subseteq \bbP^n_K$. In short, we write $R^e=R\otimes_{K[x_0]}R$.
Let $J$ be the kernel of the canonical multiplication map
$\mu: R^e\rightarrow R, f\otimes g\mapsto fg$.
Then $\Ann_{R^e}(J)$ is a homogeneous ideal of $R^e$.

\begin{definition}
The homogeneous ideal $\vartheta_N := \mu(\Ann_{R^e}(J))$ of $R$
is called the \textit{Noether different} of $\bbX$ (or of $R/K[x_0]$).
\end{definition}

More generally, for any algebra $B/A$ we can define in the same way
the Noether different 
$\vartheta_N(B/A)=\mu(\Ann_{B^e}(\widetilde{J}))$, 
where $\widetilde{J}=\Ker(B^e\stackrel{\mu}{\rightarrow} B)$. 
Similarly to the K\"ahler differents, 
the Noether differents of algebras $S/K$ and $\overline{R}/K$, 
where $S=R/\ideal{x_0-1}$ and $\overline{R}=R/\ideal{x_0}$, satisfy 
$$
\vartheta_N(S/K) = \vartheta_N/\ideal{x_0-1} 
= \prod_{j=1}^s \vartheta_N(\calO_{\bbX,p_j}/K),\quad 
\vartheta_N(\overline{R}/K) = \vartheta_N/\ideal{x_0}.
$$ 
According to \cite[Propositions 10.17-18]{Kun1986}, the relation between 
the K\"ahler and Noether differents is given by 
$\vartheta_N^n \subseteq \vartheta_\bbX\subseteq \vartheta_N$
and two differents agree when $\bbX$ is a complete intersection.

\begin{lemma}\label{lem:LocallyGor-NDiff}
\begin{enumerate}
\item[(a)] The local ring $\calO_{\bbX,p_j}$ is a local
Gorenstein ring if and only if $\vartheta_N(\calO_{\bbX,p_j}/K)\ne 0$.
In this case $\vartheta_N(\calO_{\bbX,p_j}/K)
=\Ann_{\calO_{\bbX,p_j}}(\fm_{\bbX,p_j})$.
\item[(b)] The scheme $\bbX$ is arithmetically Gorenstein if and only if 
$\vartheta_N(\overline{R}/K)\ne 0$.
\end{enumerate}
\end{lemma}
\begin{proof}
This follows from \cite[Theorem~A.5]{EU2019}.
\end{proof}

In $\bbP^2_K$, a 0-dimensional Gorenstein ring is also 
a 0-dimensional complete intersection. The lemma yields
the following immediate consequence.

\begin{corollary}
\begin{enumerate}
\item[(a)] If $\bbX$ is a 0-dimensional locally complete intersection 
in $\bbP^n_K$, then $\vartheta_\bbX(S/K)=\vartheta_N(S/K)$. 
\item[(b)] If $\bbX$ is a 0-dimensional scheme in $\bbP^2_K$, then 
$\vartheta_\bbX(S/K)=\vartheta_N(S/K)$ and 
$\vartheta_\bbX(\overline{R}/K)=\vartheta_N(\overline{R}/K)$.
\end{enumerate}
\end{corollary}

A point $p_j\in\Supp(\bbX)=\{p_1,...,p_s\}$ is called 
a \textit{Gorenstein point} of $\bbX$ if $\calO_{\bbX,p_j}$ is 
a local Gorenstein ring. Using an argument similar to that given
in the proof of Proposition~\ref{prop_HF-Kdiff}, we obtain
a formula for the Hilbert polynomial of $\vartheta_N$, as follows.

\begin{corollary}\label{cor:HF-Ndiff}
Let $\bbX_{\rm gor}$ be the set of Gorenstein points in the support 
of~$\bbX$. Then the Hilbert polynomial of $\vartheta_N$ is given by
$$
\HP_{\vartheta_N} = 
\sum_{p_j\in \bbX_{\rm gor}} \dim_K K(p_j)
= \sum_{p_j\in \bbX_{\rm gor}} \kappa_j
$$ 
In particular, $\bbX$ is locally Gorenstein if and only if
$\HP_{\vartheta_N} =\sum_{j=1}^s \dim_K K(p_j)$.
\end{corollary}

Notice that the homogeneous ring of quotients $Q^h(R)$ of $R$
is a graded-free $L_o$-module of rank $\deg(\bbX)$, where
$L_o \!=\! K[x_0,x_0^{-1}]$. Set $\omega_{Q^h(R)}\!=\!\Hom_{L_o}(Q^h(R),L_o)$. 
Based on \cite[Proposition~3.1]{KLL2019}, the scheme $\bbX$ is locally 
Gorenstein if and only if $\omega_{Q^h(R)} \cong Q^h(R)$.
In this case, a homogeneous element $\sigma \in \omega_{Q^h(R)}$
with $\omega_{Q^h(R)} = Q^h(R)\cdot \sigma$ is a 
\textit{homogeneous trace map} of $Q^h(R)/L_o$.
For further information on (canonical, homogeneous) trace maps
we refer to \cite[Appendix~F]{Kun1986}.

Let $\sigma$ be a fixed homogeneous trace map of degree zero
of $Q^h(R)/L_o$. There is a canonical monomorphism of graded $R$-modules
\begin{equation}\label{mapChaSec2.2.7}
\begin{aligned}
\Phi: \omega_R(1)
&\lhook\joinrel\longrightarrow \omega_{Q^h(R)}
= Q^h(R)\cdot\sigma \xrightarrow{\sim} Q^h(R)\\
\varphi &\longmapsto \varphi\otimes {\rm id}_{L_o}.
\end{aligned}
\end{equation}
The image of $\Phi$ is a homogeneous fractional ideal~$\fC_{\bbX}^\sigma$ 
of~$Q^h(R)$. It is also a finitely generated graded $R$-module and
$\HF_{\fC_{\bbX}^\sigma}(i) = \deg(\bbX)-\HF_\bbX(-i-1)$ for all $i\in\bbZ$.

\begin{definition}
The graded $R$-module $\fC_{\bbX}^\sigma$ is called the 
\textit{Dedekind complementary module} of $\bbX$ (or of $R/K[x_0]$)
with respect to $\sigma$. Its inverse 
$$
\vartheta_D^\sigma = \{\, f \in Q^h(R) \,\mid\, f\cdot\fC_{\bbX}^\sigma
\subseteq R\,\}
$$
is called the \textit{Dedekind different} of $\bbX$ (or of~$R/K[x_0]$).
\end{definition}

The Dedekind different $\vartheta_D^\sigma$ of $\bbX$ is a homogeneous
ideal of $R$ and has the following basic properties (see \cite[Propositions~3.7-8]{KL2017} and \cite[Corollary~3.6]{KLL2019}).

\begin{proposition}\label{prop:HF-DDiff}
Let $\bbX\subseteq \bbP^n_K$ be a 0-dimensional locally Gorenstein
scheme, and let $\sigma$ be a homogeneous trace map of degree zero
of $Q^h(R)/L_o$.
\begin{enumerate}
\item[(a)] We have $\HP_{\vartheta_D^\sigma} = \deg(\bbX)$ and
$r_\bbX\le \ri(\vartheta_D^\sigma)\le 2r_\bbX$.
\item[(b)] If $\bbX$ is reduced, then $(\vartheta_D^\sigma)^n \subseteq 
\vartheta_\bbX \subseteq \vartheta_N = \vartheta_D^\sigma$ 
and $\ri(\vartheta_N)=\ri(\vartheta_D^\sigma)=2r_\bbX$.
\end{enumerate}
\end{proposition}

\begin{remark}\label{rem:N-D-DiffDiagram}
Let us discuss more about the relation between the Dedekind and Noether
differents of a 0-dimensional locally Gorenstein scheme 
$\bbX\subseteq\bbP^n_K$. Let 
$\mu_{Q^h(R)}: (Q^h(R))^e= Q^h(R)\otimes_{L_o}Q^h(R) \rightarrow Q^h(R),
f\otimes g \mapsto fg$ and $\mathcal{J}=\Ker(\mu_{Q^h(R)})$.
By~\cite[F.9]{Kun1986}, this is an isomorphism of graded $R$-modules
$$
\Theta: \Ann_{(Q^h(R))^e}(\mathcal{J}) \stackrel{\sim}{\longrightarrow}
\Hom_{Q^h(R)}(\omega_{Q^h(R)}, Q^h(R))
$$
given by $f= \sum_k f_k\otimes g_k \mapsto \Theta(f)$
with $\Theta(f)(\varphi) = \sum_k \varphi(f_k)g_k$
and $\varphi\in \omega_R(1)$. Let ${\rm Tr}:= {\rm Tr}_{Q^h(R)/L_o}$
be the canonical trace map of $Q^h(R)/L_o$ and consider 
the diagram
\[
\xymatrix{
\Ann_{R^e}(J)  \ar[rr]^-{\Theta} \ar@{^{(}->}[d]
&& \Hom_{R}(\omega_{R}(1),R)
\ar@{^{(}->}[d] \\
\Ann_{(Q^h(R))^e}(\mathcal{J})  \ar[rr]^-{\Theta} \ar[d]_{\mu_{Q^h(R)}}
&& \Hom_{Q^h(R)}(\omega_{Q^h(R)}, Q^h(R)) \ar[d]^-{\Psi_{\sigma}}
\ar[lld]_-{\Psi_{\rm Tr}}\\
Q^h(R) \ar@{=}[rr] && Q^h(R)
}
\]
where $\Psi_{\sigma}$ is an isomorphism given by
$\Psi_{\sigma}(\Theta(f))=\Theta(f)(\sigma)$ and where $\Psi_{\rm Tr}$ 
is a homomorphism with $\Psi_{\rm Tr}(\Theta(f))= \Theta(f)({\rm Tr})$
for $f\in\Ann_{L^e}(\mathcal{J})$. Then we have 
$\mu_{Q^h(R)} = \Psi_{\rm Tr}\circ\Theta$. Hence
$\vartheta_N = \Psi_{\rm Tr}\circ\Theta(\Ann_{R^e}(J))$
and $\vartheta_D^\sigma =\Psi_\sigma\circ\Theta(\Ann_{R^e}(J))$.
In particular, if ${\rm Tr}$ is a trace map of $Q^h(R)/L_o$, then
$\vartheta_N=\vartheta_D^{\rm Tr}$.
\end{remark}

Using the Dedekind and Noether differents, we can characterize
0-dimensional arithmetically Gorenstein schemes as follows.

\begin{proposition}\label{CharGorNoetDiff}
Let $\bbX$ be a 0-dimensional scheme in $\bbP^n_K$.
The following assertions are equivalent.
\begin{itemize}
  \item[(a)] $\bbX$ is arithmetically Gorenstein.
  \item[(b)] $\bbX$ is a locally Gorenstein CB-scheme 
  and $\HF_{\vartheta_D^\sigma}(r_\bbX )\ne 0$,
  where $\sigma$ is a homogeneous trace map of $Q^h(R)/L_o$
  of degree zero.
  \item[(c)] $\bbX$ is a locally Gorenstein CB-scheme 
  and $\HF_{\vartheta_N}(r_\bbX)\ne 0$.
  \item[(d)] $\bbX$ is a CB-scheme and
  $\HP_{\vartheta_N} = \sum_{j=1}^s\dim_K K(p_j)$ 
  and $\HF_{\vartheta_N}(r_\bbX)\ne 0$.
\end{itemize}
If these conditions are satisfied, then $\vartheta_N$ is a principal ideal
generated by a non-zero homogeneous element of degree $r_\bbX$ and
$\ri(\vartheta_N)\le 2r_\bbX$.
\end{proposition}
\begin{proof}
The equivalence of (a) and (b) follows from \cite[Proposition~5.8]{KL2017}.
Now we prove the implication ``(a)$\Rightarrow$(c)''. 
Suppose that $\bbX$ is arithmetically Gorenstein.  
It follows at once that $\bbX$ is a locally Gorenstein CB-scheme.
By \cite[E.16]{Kun1986}, there exists a homogeneous trace map 
$\tilde{\sigma}$ of degree $-r_\bbX$ of~$R/K[x_0]$ such that 
$\omega_{R}(1)=R\cdot\tilde{\sigma}$.
Since $R$ is a graded-free $K[x_0]$-algebra of rank $d:=\deg(\bbX)$,
we let $B=\{b_1,...,b_d\}$ be a homogeneous $K[x_0]$-basis of $R$, 
let $\{b'_1,...,b'_d\}$ be the dual basis of $B$ with respect to 
$\widetilde{\sigma}$, i.e., $\widetilde{\sigma}(b'_ib_j)=\delta_{ij}$ 
for all $i,j \in \{1,...,d\}$, and let $\Delta_{\widetilde{\sigma}}
= \sum_{i=1}^d b'_i\otimes b_i$.
By \cite[F.10]{Kun1986}, the graded $R$-module $\Ann_{R^e}(J)$ 
is free of rank~$1$ and it is generated by~$\Delta_{\tilde{\sigma}}$.
So, $\vartheta_N$ is a principal ideal generated 
by~$\mu(\Delta_{\tilde{\sigma}})$.
Since ${\rm char}(K) \nmid d$,  the canonical trace map of $R/K[x_0]$ 
satisfies ${\rm Tr}_{R/K[x_0]}(1) = d \ne 0$.  
It follows from \cite[F.12]{Kun1986} that 
${\rm Tr}_{R/K[x_0]}=\mu(\Delta_{\tilde{\sigma}})\cdot\tilde{\sigma}$,
and hence $\mu(\Delta_{\tilde{\sigma}}) \ne 0$.
Moreover, by \cite[F.4]{Kun1986}, ${\rm Tr}_{R/K[x_0]}=\sum_{i=1}^d
b_i\cdot b_i^*$, where $\{b_1^*,...,b_d^*\}\subseteq \omega_R$ 
is the dual basis of $B$. 
Since $\deg(b_i\cdot b_i^*)=0$ for $i=1,...,d$, 
${\rm Tr}_{R/K[x_0]}$ is homogeneous of degree zero. Therefore
we get $\deg(\mu(\Delta_{\tilde{\sigma}})) = r_\bbX$.

Next, we prove the implication ``(c)$\Rightarrow$(b)''.
For this, it suffices to show that $\HF_{\vartheta_D^\sigma}(r_\bbX) \ne 0$.
Since $\HF_{\vartheta_N}(r_\bbX)\ne 0$, we find a non-zero homogeneous 
element $f\in (\Ann_{R^e}(J))_{r_\bbX}$. 
By Remark~\ref{rem:N-D-DiffDiagram}, the map $\Psi_{\sigma}\circ\Theta: \Ann_{R^e}(J)\rightarrow \vartheta_D^\sigma$ is an isomorphism of graded
$R$-modules. Hence $\Psi_{\sigma}\circ\Theta(f) \in
\vartheta_D^\sigma\setminus\{0\}$.
Furthermore, the equivalence of (c) and (d) follows by 
Corollary~\ref{cor:HF-Ndiff}.

Finally, the additional claim for the regularity index
of $\vartheta_N$ follows from the fact that the $R$-linear map 
$\widetilde{\imath}: R\rightarrow S[x_0]$
sends $(\vartheta_N)_{r_\bbX}\cdot R_{r_\bbX}$ 
onto $\vartheta_N(S/K)x_0^{2r_\bbX}$ and $\HP_{\vartheta_N}=\dim_K(\vartheta_N(S/K))$.
\end{proof}

According to Proposition~\ref{prop:HF-DDiff} and 
\cite[Proposition~5.8]{KL2017}, a reduced 0-dimensional scheme $\bbX$ 
is arithmetically Gorenstein if and only if its Noether different 
is a principal ideal. In this case, we have $\ri(\vartheta_N)=2r_\bbX$.
However, when $\bbX$ is not reduced, the 
Noether different of $\bbX$ can be a principal ideal
without $\bbX$ being an arithmetically Gorenstein scheme.

\begin{example}
\begin{enumerate}
\item[(a)] Let $\bbX$ be the 0-dimensional scheme in $\bbP^2_{\bbQ}$
defined by the ideal $I_\bbX=\ideal{F_1,F_2,F_3}$, where 
$F_1= X_{2}^2$, $F_2=X_0X_{2} -X_{1}X_{2}$, and
$F_3=X_0^4  -2X_0^3 X_{1} +2X_0^2 X_{1}^2  -2X_0X_{1}^3  +X_{1}^4$.
Then $\deg(\bbX)=5$, $\Supp(\bbX)=\{p_1,p_2\}$ and
$I_\bbX = I_{p_1}\cap I_{p_2}$ with $I_{p_1} = (\ideal{X_1-X_0,X_2})^2$
and $I_{p_2} = \ideal{X_1^2+X_0^2,X_2}$. Also, the Hilbert function
of $\bbX$ is $\HF_\bbX:\ 1\ 3\ 4\ 5\ 5\cdots$ and $r_\bbX=3$.
Since $\HF_\bbX$ is not symmetric, the scheme $\bbX$ is 
not arithmetically Gorenstein. However, the Noether different of $\bbX$
is a principal ideal given by $\vartheta_N=
\ideal{ x_0^3 x_1 -3 x_0^2 x_1^2  +3 x_0 x_1^3  -x_1^4}$.
Its Hilbert function is $\HF_{\vartheta_N}:\ 0\ 0\ 0\ 0\ 1\ 2\ 2\cdots$,
and so $\HF_{\vartheta_N}(r_\bbX)=0$ and $\ri(\vartheta_N)=5<2r_\bbX$.

\item[(b)] Consider the 0-dimensional scheme $\bbX\subseteq\bbP^3_{\bbQ}$
defined by the ideal $I_\bbX=\ideal{F_1,...,F_5}$, where
$F_1= X_0X_3 -2X_1X_3 -X_{3}^2$, $F_2=X_{2}^2 -X_{1}X_{3} -2X_{2}X_{3}$,
$F_3=X_{1}X_{2} +X_{2}X_{3}$, $F_4=X_{0}X_{2} +X_{2}X_{3}$, and
$F_5=X_{1}^2 +X_{1}X_{3}$. We have $\deg(\bbX)=5$, 
$\Supp(\bbX)=\{p_1,p_2,p_3\}$ and $I_\bbX=I_{p_1}\cap I_{p_2}\cap I_{p_3}$
with $I_{p_1}=\ideal{X_1,X_2,X_3-X_0}$, 
$I_{p_2}=\ideal{X_1-X_0,(X_2-X_0)^2,X_3-X_0}$,
and $I_{p_3}=\ideal{X_1^2,X_2,X_3}$. We also have
$\HF_\bbX:\ 1\ 4\ 5\ 5\cdots$, $r_\bbX=2$, and 
$\HF_{\vartheta_N}:\ 0\ 0\ 1\ 3\ 3\cdots$. Since $\HP_{\vartheta_N}=3$ 
and $\deg_\bbX(p_i)=2$ for $i=1,2,3$, the scheme $\bbX$ is a locally 
Gorenstein CB-scheme. Thus the condition $\HF_{\vartheta_N}(r_\bbX)
=1\ne 0$ yields that $\bbX$ is an arithmetically Gorenstein scheme.
In addition, the Noether different is a principal ideal generated by
$2 x_{0} x_{1} +3 x_{1} x_{3} +2 x_{2} x_{3} -x_{3}^2$ and 
$\ri(\vartheta_N)=3<2r_\bbX$.
\end{enumerate}
\end{example}

\begin{corollary}
If $\bbX\subseteq\bbP^n_K$ is a 0-dimensional locally Gorenstein CB-scheme,
then we have $\HF_{\vartheta_\bbX}(i)=\HF_{\vartheta_N}(i)=0$ for all 
$i < r_\bbX$.
\end{corollary}
\begin{proof}
By the assumption, the Dedekind different $\vartheta_D^\sigma$ 
is well-defined with a homogeneous trace map $\sigma$ of $Q^h(R)/L_o$
of degree zero. By \cite[Proposition~4.8]{KLL2019}, we have 
$\HF_{\vartheta_D^\sigma}(r_\bbX-1)=0$. It follows from 
Remark~\ref{rem:N-D-DiffDiagram} that $\HF_{\Ann_{R^e(J)}}(r_\bbX-1)=0$,
and subsequently $\HF_{\vartheta_N}(r_\bbX-1)=0$. Thus the claim follows 
from $\vartheta_\bbX\subseteq \vartheta_N$.
\end{proof}

Now we are ready to state and prove a characterization of 0-dimensional 
complete intersections in $\bbP^n_K$.

\begin{thm}\label{thm:CharacterizationCI-KDiff}
Let $\bbX$ be a 0-dimensional scheme in $\bbP^n_K$. Then
the following conditions are equivalent:
\begin{enumerate}
\item[(a)] $\bbX$ is a complete intersection.
\item[(b)] $\bbX$ is a locally Gorenstein CB-scheme and 
$\HF_{\vartheta_\bbX}(r_\bbX) \ne 0$.
\item[(c)] $\bbX$ is a CB-scheme and 
$\HP_{\vartheta_\bbX} = \sum_{j=1}^s\dim_K K(p_j)$ and
$\HF_{\vartheta_\bbX}(r_\bbX) \ne 0$.
\end{enumerate}
\end{thm}
\begin{proof}
We first prove the equivalence of (a) and (b).
If $\bbX$ is a complete intersection, then it is also
an arithmetically Gorenstein scheme. By Proposition~\ref{CharGorNoetDiff}, 
$\bbX$ is a locally Gorenstein CB-scheme and the Noether different 
$\vartheta_N$ satisfies $\HF_{\vartheta_N}(r_\bbX)\ne 0$. 
In this case, we have $\vartheta_\bbX = \vartheta_N$ by 
\cite[G.3 and 10.2]{Kun1986}. Hence $\HF_{\vartheta_\bbX}(r_\bbX) \ne 0$.

Conversely, the condition $\HF_{\vartheta_\bbX}(r_\bbX) \ne 0$ implies 
$\HF_{\vartheta_N}(r_\bbX)\ne 0$, as $\vartheta_\bbX \subseteq\vartheta_N$.  
So, it follows from Proposition~\ref{CharGorNoetDiff} that $\bbX$ is
an arithmetically Gorenstein scheme.
In particular, $\vartheta_N$ is a principal ideal generated 
by a non-zero homogeneous element $h$ of degree $r_\bbX$. More precisely,
as in the proof of Proposition~\ref{CharGorNoetDiff} we have 
$h= \mu(\Delta_{\tilde{\sigma}})$, where $\tilde{\sigma}$ is a homogeneous
trace map of degree $-r_\bbX$ of the algebra~$R/K[x_0]$, 
and where $\Delta_{\tilde{\sigma}} = \sum_{i=1}^d b'_i\otimes b_i$ 
with a homogeneous $K[x_0]$-basis $\{b_1,\dots,b_d\}$ of $R$ and its dual
basis $\{b'_1,...,b'_d\}$ with respect to ${\tilde{\sigma}}$, $d=\deg(\bbX)$.
Here we may assume that $b_1 = 1$. Now we need to verify that 
$\mu(\Delta_{\tilde{\sigma}}) \notin \ideal{x_0} \subseteq R$.
Let ${\rm Tr}_{R/K[x_0]}$ be the canonical trace map of $R/K[x_0]$.
Then ${\rm Tr}_{R/K[x_0]}=\mu(\Delta_{\tilde{\sigma}})\cdot\tilde{\sigma}$.
Also, it is straightforward to check that 
$(\sum_{i=1}^d {\rm Tr}_{R/K[x_0]}(b_i)b'_i)\cdot\tilde{\sigma}
={\rm Tr}_{R/K[x_0]}$. 
It follows that
$$
\mu(\Delta_{\tilde{\sigma}}) =
\sum_{i=1}^d {\rm Tr}_{R/K[x_0]}(b_i)b'_i
= d\cdot b'_1 + \sum_{i=2}^{d} {\rm Tr}_{R/K[x_0]}(b_i)b'_i
$$
Clearly, we have $\deg(b'_1)=\deg(\mu(\Delta_{\tilde{\sigma}}))=r_\bbX$
and $d\ne 0$, as ${\rm char}(K)\nmid d$.
Note that the residue classes of $b'_1,...,b'_d$ also form a homogeneous
$K$-basis of $\overline{R} = R/\ideal{x_0}$.  In $\overline{R}$ 
we have
$$
\overline{\mu(\Delta_{\tilde{\sigma}})}
= d\cdot \overline{b'_1} +
\sum_{i=2}^{d} \overline{{\rm Tr}_{R/K[x_0]}(b_i)}\overline{b'_i}
\ne 0
$$
Hence we get $\vartheta_\bbX(\overline{R}/K)= 
\vartheta_\bbX/\ideal{x_0} \ne 0$, and consequently an application of
Proposition~\ref{prop:cipoint-ci}.b yields that the scheme $\bbX$
is a complete intersection, as desired.

Finally, the equivalence of (a) and (c) follows by Corollary~\ref{cor:LocallyCI} and from the above argument and
the fact that every locally complete intersection is locally Gorenstein.
\end{proof}

\begin{remark}
When the scheme $\bbX$ is a smooth 0-dimensional scheme, our result recovers 
a result of~\cite{KL2017}.
\end{remark}

We end this section with an application of the theorem to
a concrete example.

\begin{example}\label{twocubics}
Let $\bbX$ be the 0-dimensional scheme in $\bbP^2_\bbQ$ defined by
the ideal $I_\bbX =I_{p_1}\cap I_{p_2}\cap I_{p_3}\cap I_{p_4}$,
where $I_{p_1} = \ideal{X_1, X_2}$, $I_{p_2}=\ideal{X_1-4X_2, X_2-2X_0}$,
$I_{p_3}= \ideal{X_1 -3X_2-2X_0, (X_2-X_2)^2}$, and 
$$
I_{p_4} =\ideal{X_0^2X_1-X_2^3, X_1^2-X_0X_2-X_2^2, 
X_0^3+X_0^2X_2-X_1X_2^2}
$$
Here $p_1,p_2$ are $\bbQ$-rational points, $p_4$ is a reduced point, 
while $p_3$ is a non-reduced point.
We have $\HF_\bbX:\ 1\ 3\ 6\ 8\ 9\ 9\cdots$, $r_\bbX=4$,
$\HF_{\vartheta_N}=\HF_{\vartheta_\bbX}:\ 0\ 0\ 0\ 0\ 1\ 3\ 6\ 8\ 8\cdots$
and $\ri(\vartheta_N) = \ri(\vartheta_\bbX)=7$.
Since $\HP_{\vartheta_N}=8=\sum_{j=1}^4\dim_\bbQ \bbQ(p_j)$, the scheme 
$\bbX$ is locally Gorenstein. Moreover, we have 
$\deg_\bbX(p_j)=4$ for $j=1,...,4$, and so $\bbX$ is a CB-scheme.
Therefore, by Theorem~\ref{thm:CharacterizationCI-KDiff},
$\HF_{\vartheta_\bbX}(r_\bbX)=1\ne 0$ implies that $\bbX$ is a
complete intersection. 
\end{example}

\subsection*{Acknowledgments.}
The author thanks Martin Kreuzer and Lorenzo Robbiano for their encouragement
to elaborate some results presented here.

\bigbreak

\end{document}